\documentclass[11pt]{article}
\usepackage{amsfonts}
\usepackage{amsfonts}

\usepackage{amsfonts}
\usepackage{}
\usepackage{bbm}
\usepackage [dvips]{graphics}
\usepackage[centertags]{amsmath}
\usepackage{amsfonts}
\usepackage{amssymb}
\usepackage{amsthm}
\usepackage{newlfont}
\usepackage{stmaryrd}
\usepackage[]{titletoc}  
\usepackage{mathrsfs}
\usepackage{stmaryrd}
\usepackage [dvips]{graphics}
\usepackage[centertags]{amsmath}
\usepackage{amsfonts}
\usepackage{amssymb}
\usepackage{amsthm}
\usepackage{newlfont}
\usepackage{hyperref}

\newlength{\defbaselineskip}
\newcommand{\setlinespacing}[1]%
           {\setlength{\baselineskip}{#1 \defbaselineskip}}

\theoremstyle{plain}
\newtheorem{thm}{Theorem}[section]
\newtheorem{cor}[thm]{Corollary}
\newtheorem{lem}[thm]{Lemma}

\theoremstyle{definition}
\newtheorem{defn}{Definition}[section]

\newtheorem{rem}{Remark}[section]

\newcommand{\mR}{\mathbb{R}}
\newcommand{\eps}{\varepsilon}
\newcommand{\la}{\langle}
\newcommand{\ra}{\rangle}

\newcommand{\mL}{\mathcal {L}}
\newcommand{\mF}{\mathcal {F}}
\newcommand{\mE}{\mathcal {E}}
\newcommand{\mP}{\mathcal {P}}
\newcommand{\mB}{\mathcal {B}}
\newcommand{\pat}{\partial}
\newcommand{\mY}{\mathcal {Y}}
\newcommand{\mX}{\mathcal {X}}
\newcommand{\mM}{\mathcal {M}}
\newcommand{\mA}{\mathcal {A}}
\newcommand{\wtil}{\widetilde}
\makeatletter\@addtoreset{equation}{section} \makeatother
\allowdisplaybreaks[4]
\begin{document}
\title {Semi-linear Backward Stochastic Integral Partial Differential Equations driven by a Brownian motion and a Poisson point process}
\author{Shaokuan Chen, \ \ Shanjian Tang }
\maketitle  \noindent \textbf{Abstract:} In this paper we
investigate classical solution of a semi-linear system of backward
stochastic integral partial differential equations driven by a
Brownian motion and a Poisson point process. By proving an
It\^{o}-Wentzell formula for jump diffusions as well as an abstract
result of stochastic evolution equations, we obtain the stochastic
integral partial differential equation for the inverse of the
stochastic flow generated by a stochastic differential equation
driven by a Brownian motion and a Poisson point process. By
composing the random field generated by the solution of a backward
stochastic differential equation with the inverse of the stochastic
flow, we construct the classical solution of the system of backward
stochastic integral partial differential equations. As a result, we
establish a stochastic Feynman-Kac
formula.
\footnotetext[1]{Department of Financial Mathematics and Control
Science, School of Mathematical Sciences, Fudan University, Shanghai
200433, China. Research supported by NSFC under grant No. 10325101,
Basic Research Program of China (973 Program) under grant No.
2007CB814904. E-mail: skchen0118@gmail.com (Shaokuan Chen),
sjtang@fudan.edu.cn (Shanjian Tang).}
\section{Introduction}
Backward stochastic partial differential equations (BSPDEs) are
function space-valued backward stochastic differential equations
(BSDEs), the theories and applications of which can be found in
 \cite{BaBuPa}, \cite{Bismut}, \cite{DeTa}, \cite{KPQ}, \cite{PP1},
\cite{PP2}, \cite{YZ}, etc. BSPDEs arise from the study of
stochastic control problems. For example, they can serve as the
adjoint equations of Duncan-Mortensen-Zakai equation in the optimal
control of stochastic differential equations (SDEs) with incomplete
information (see \cite{Ben}, \cite{Tang2}, \cite{Tang3}). They also
appear as the adjoint equations in the stochastic maximum principle
of systems governed by stochastic partial differential equations
(SPDEs) driven by a Brownian motion (see \cite{Zhou2}) or driven by
both a Brownian motion and a Poisson random measure (see
\cite{OkPrZh}). A class of fully nonlinear BSPDEs, the so-called
backward stochastic Hamilton-Jacobi-Bellman (HJB) equations, were
proposed by Peng \cite{P2} in the study of the optimal control
problems for non-Markovian cases. And Englezos and Karatzas
\cite{EnKa} characterized the value function of a utility
maximization problem with habit formation as a classical solution of
the corresponding stochastic HJB equation, which gives a concrete
illustration of BSPDEs in a stochastic control context beyond the
classical linear quadratic case. More recently, Meng and Tang
\cite{MeTa} have studied the maximum principle of non-Markovian
stochastic differential systems driven only by Poisson point
processes and obtained a kind of backward stochastic HJB equations
of jump type.

The existence, uniqueness and regularity of the adapted solutions to
BSPDEs have been studied by many authors. See Hu and Peng \cite{HP},
Peng \cite{P2}, Zhou \cite{Zhou1}, Ma and Yong \cite{MY1},
\cite{MY2}, Hu, Ma and Yong \cite{HMY}, for non-degenerate and
degenerate cases. More recently, Du and Tang \cite{DuTa} explored
the Dirichlet problem, rather than the conventional Cauchy problem,
of the BSPDEs and established the result of existence and uniqueness
in weighted Sobolev spaces. Different from the above literatures
where the main instrument is operator semigroup or a prior estimates
for differential operators, Tang \cite{Tang} developed a
probabilistic approach to study the properties of solutions of
BSPDEs. To be precise, he constructed the solutions of BSPDEs in
terms of the inverse flows of the solutions of SDEs as well as the
solutions of BSDEs. As a result, the properties of the solutions of
BSPDEs can be obtained by the analysis of the solutions of SDEs and
BSDEs with a spacial parameter.

Let $(\Omega,\mathcal {F},\{\mathcal {F}_t\}_{0\leq t\leq T},P)$ be
a complete filtered probability space on which is defined a
$d$-dimensional standard Brownian motion $\{W_t,0\leq t\leq T\}$.
Denote by $\mathcal {P}$ the predictable sub-$\sigma$-field of
$\mathcal {B}([0,T])\otimes\mF$. Let $(E,\mE,v)$ be a measure space
with $v(E)<\infty$ and
$\mathbbm{p}:D_\mathbbm{p}\subset(0,\infty)\rightarrow E$ be an
$\mathcal {F}_t$-adapted stationary Poisson point process. Then the
counting measure induced by $\mathbbm{p}$ is defined by
$$N((0,t]\times U):=\sharp\{s\in D_\mathbbm{p};s\leq t, \mathbbm{p}(s)\in U\},\
\textrm{for}\ t>0,\ U\in\mE.$$ And
$\widetilde{N}(dedt):=N(dedt)-v(de)dt$ is a compensated Poisson
random measure which is assumed to be independent of the Brownian
motion.

This paper is concerned with the probabilistic interpretation of the
solution $(p,q,r)$ of the following system of backward stochastic
integral partial differential equations (BSIPDEs)
\begin{equation}
\left\{
\begin{array}{ccc}\label{a1}
\displaystyle
\begin{split}
  dp(t,x)=&-\big[\mathcal {L}(t,x)p(t-,x)+\mathcal
  {M}(t,x)q(t,x)\\
  &+f(t,x,p(t,x),q(t,x)+\partial
  p(t-,x)\sigma,r(t,\cdot,\phi_{t,\cdot}(x))-p(t-,x)+p(t-,\phi_{t,\cdot}(x)))\big]dt\\
  &+\int_E\big[r(t,e,x)-r(t,e,\phi_{t,e}(x))+p(t-,x)-p(t-,\phi_{t,e}(x))\big]v(de)dt\\
  &+q(t,x)dW_t+\int_Er(t,e,x)\widetilde{N}(dedt),\ (t,x)\in[0,T]\times\mathbb{R}^n,\\
  p(T,x)=&\varphi(x),
\end{split}
\end{array}
\right.
\end{equation}
where we have defined
\begin{eqnarray}
\begin{split}\label{oper}
  &\pat_i:=\frac{\pat}{\pat x^i},\ \pat^2_{ij}:=\frac{\pat^2}{\pat x^i\pat
  x^j},\ \pat p:=(\pat_jp^i)_{1\leq i\leq l,1\leq j\leq n},\\
  &\mL(t,x):=\frac{1}{2}\sum_{i,j=1}^{n}\sum_{r=1}^{d}\sigma^{ir}\sigma^{jr}(t,x)\pat^2_{ij}+\sum_{i=1}^{n}\big[b^i(t,x)-\int_Eg^i(t,e,x)v(de)\big]\pat_i,\\
  &\mM^r(t,x):=\sum_{i=1}^n\sigma^{ir}(t,x)\pat_i,\ 1\leq r\leq d,\\
  &\phi_{t,e}(x):=x+g(t,e,x),\\
  &\mathcal{L}p:=(\mathcal {L}p^1,\cdots,\mathcal {L}p^l)',\\
  &\mathcal {M}q:=(\sum_{r=1}^d\mathcal {M}^rq^{1r},\cdots,\sum_{r=1}^d\mathcal
  {M}^rq^{lr})'.\\
\end{split}
\end{eqnarray}
To be precise, under some suitable conditions on the random
coefficients $b,\ \sigma,\ g$ and $\varphi$, we can construct the
solution $(p,q,r)$ of the above system as follows
\begin{equation*}
\begin{split}\label{a2}
  p(t,x)=&Y_t(X_t^{-1}(x)),\\
  q(t,x)=&Z_t(X_{t-}^{-1}(x))-\pat p(t-,x)\sigma(t,x),\\
  r(t,e,x)=&p(t-,\phi_{t,e}^{-1}(x))-p(t-,x)+U_t(e,X_{t-}^{-1}(\phi_{t,e}^{-1}(x)))\\
\end{split}
\end{equation*}
where $(X_t(x),Y_t(x),Z_t(x),U_t(\cdot,x))_{t\in[0,T]}$ is the
solution of the following non-Markovian forward-backward stochastic
differential equation (FBSDE)
\begin{equation*}
\left\{
\begin{split}
dX_t&=b(t,X_t)dt+\sigma(t,X_t)dW_t+\int_Eg(t,e,X_{t-})\wtil{N}(dedt),\\
dY_t&=-f(t,X_t,Y_t,Z_t,U_t)dt+Z_tdW_t+\int_EU_t(e)\wtil{N}(dedt),\\
X_0&=x,\ Y_T=\varphi(X_T),\ t\in[0,T],
\end{split}
\right.
\end{equation*}
and $\{X_t^{-1}(x),\ (t,x)\in[0,T]\times\mathbb{R}^n\}$ is the
inverse mapping of $\{X_t(x),\ (t,x)\in[0,T]\times\mathbb{R}^n\}$
with respect to the spacial variable $x$. The main contributions of
this paper lie in establishing an It\^{o}-Wentzell formula for jump
diffusions and deriving the SIPDE satisfied by the inverse flow
$\{X_t^{-1}(x),\ (t,x)\in[0,T]\times\mathbb{R}^n\}$. As demonstrated
in Tang \cite{Tang}, the above two facts play crucial roles in the
construction of the classical solutions to BSPDEs. Moreover, by the
analysis of solutions of BSDEs driven by a Brownian motion and a
Poisson point process, we generalize Tang's result to the jump case.

The paper is organized as follows. Section 2 contains the notations
that are used throughout the paper and some preliminary results.
Section 3 is concerned with the derivation of the SIPDE for the
inverse flow $\{X_t^{-1}(x),\ (t,x)\in[0,T]\times\mathbb{R}^n\}$.
Section 4 consists of the estimates of the solutions of BSDEs.
Section 5 is constituted the connection between BSIPDEs and
non-Markovian FBSDEs.
\section{Notations and Preliminary results}
Let $\mathbb{E}$ be a Euclidean space. The inner product in
$\mathbb{E}$ is denoted by $\la\cdot,\cdot\ra$ and the norm in
$\mathbb{E}$ is denoted by $|\cdot|_{\mathbb{E}}$ or simply by
$|\cdot|$ when there is no confusion. Let
$\gamma:=(\gamma^1,\cdots,\gamma^n)$ be a multi-index with
nonnegative integers $\gamma^i$, $i=1,\cdots,n$. Denote
$|\gamma|:=\gamma^1+\cdots+\gamma^n$. For a function $u$ defined on
$\mR^n$, $\pat^\gamma u$ means the derivative of $u$ of order
$|\gamma|$, of order $\gamma^i$ with respect to $x^i$. $u_i$ means
the derivative of $u$ with respect to $x^i$ and $u_{ij}$ means the
derivative of $u$ with respect to $x^ix^j$, respectively. $\pat u$
stands for the the gradient of $u$ and $\pat^2u$ stands for the
hessian of $u$, respectively. We use the convention that repeated
indices imply summation.

Let $\mathbb{B}$ be a Banach space with the norm
$\|\cdot\|_\mathbb{B}$. For an integer $m\geq0$ and some subset
$K\subseteq\mR^n$, we denote by $C^m(K;\mathbb{B})$ the set of
mappings $f:K\rightarrow \mathbb{B}$ which are $m$-times
continuously differentiable. And denote the norm in this space by
$$\|f\|_{C^m(K;\mathbb{B})}:=\sum_{0\leq|\gamma|\leq m}\sup_{x\in K}\|\pat^\gamma f(x)\|_\mathbb{B}.$$
If there is no danger of confusion, $C^m(K;\mathbb{B})$ will be
abbreviated as $C^m(K)$. Denote by $D(0,T;\mathbb{B})$ the set of
all $\mathbb{B}$-valued cadlag functions on the interval $[0,T]$.

For $p>1$ and integer $m\geq0$, we denote by $W^m_p$ the Sobolev
space of real functions on $\mR^n$ with a finite norm
$$\|u\|_{m,p}:=\bigg(\sum_{0\leq|\gamma|\leq m}\int_{\mR^n}|\pat^\gamma
u|^pdx\bigg)^{1/p}.$$ The inner product and norm in $W^m_2$ will be
denoted by $(\cdot,\cdot)_m$ and $\|\cdot\|_m$, respectively.

We further introduce some other spaces that will be used in the
paper. Let $\mathcal {X}$ (resp., $\mathcal {Y}$) be a
sub-$\sigma$-algebra of $\mF$ (resp., $\mE$). $L^p(\mX;\mathbb{B})$
(resp., $L^p(\mY;\mathbb{B})$) denotes the set of all
$\mathbb{B}$-valued $\mX$-measurable (resp., $\mY$-measurable)
random variable $\eta$ such that $E\|\eta\|_\mathbb{B}^p<\infty$
(resp., $\int_E\|\eta\|_\mathbb{B}^pv(de)<\infty$). For given two
real numbers $1\leq p,\ k\leq\infty$, we denote by
$\mL^p_{\mF}(0,T;L^k(\mF;\mathbb{B}))$ the set of all adapted
$\mathbb{B}$-valued processes $X$ such that
$$\|X\|_{\mL^p_{\mF}(0,T;L^k(\mF;\mathbb{B}))}:=\bigg(\int_0^T[E\|X(t)\|_{\mathbb{B}}^k]^{p/k}dt\bigg)^{1/p}<\infty.$$
When $k=p$, $\mL^p_{\mF}(0,T;L^k(\mF;\mathbb{B}))$ will be
abbreviated as $\mL^p_{\mF}(0,T;\mathbb{B})$. Denote by
$\mL^{\infty,p}_\mF(0,T;\mathbb{B})$ (resp.,
$\mL^{\infty,p}_{\mF,w}(0,T;\mathbb{B})$) the Banach space of all
adapted $\mathbb{B}$-valued strongly (resp., weakly) cadlag
processes $X$ for which
$$\|X\|_{\mL^{\infty,p}_\mF(0,T;\mathbb{B})}:=\bigg(E\sup_{0\leq t\leq T}\|X(t)\|_\mathbb{B}^p\bigg)^{1/p}<\infty.$$
$$\bigg(\textrm{resp.,}\ \|X\|_{\mL^{\infty,p}_{\mF,w}(0,T;\mathbb{B})}:=\bigg(E\sup_{0\leq t\leq T}\|X(t)\|_\mathbb{B}^p\bigg)^{1/p}<\infty.\bigg)$$
Denote by $\mL^{k,p}_\mF(0,T;\mathbb{B})$ (resp.,
$\mL^{k,p}_\mF([0,T]\times E;\mathbb{B})$) with $k\in[1,\infty)$ the
set of all $\mP$ measurable (resp., $\mP\otimes\mE$ measurable)
$\mathbb{B}$-valued processes $X$ such that
$$\|X\|_{\mL^{k,p}_\mF(0,T;\mathbb{B})}:=\bigg(E\bigg(\int_0^T\|X(t)\|^k_\mathbb{B} dt\bigg)^{p/k}\bigg)^{1/p}<\infty.$$
$$\bigg(\textrm{resp.,}\ \|X\|_{\mL^{k,p}_\mF([0,T]\times E;\mathbb{B})}:=\bigg(E\bigg(\int_0^T\int_E\|X(t,e)\|^k_\mathbb{B}
v(de)dt\bigg)^{p/k}\bigg)^{1/p}<\infty.\bigg)$$ Obviously when
$k=p$, $\mL^p_{\mF}(0,T;L^k(\mF;\mathbb{B}))$ coincides with
$\mL^{k,p}_\mF(0,T;\mathbb{B})$.

We consider the following SDE
\begin{equation}\label{b1}
  \left\{
  \begin{array}{ccccccccc}
  \begin{split}
  &dX_t=b(t,X_t)dt+\sigma(t,X_t)dW_t+\int_Eg(t,e,X_{t-})\wtil{N}(dedt),\\
  &X_s=x,\ t\in[s,T],
  \end{split}
  \end{array}
  \right.
\end{equation}
where $b:[0,T]\times\Omega\times\mR^n\rightarrow\mR^n$ and $\sigma:
[0,T]\times\Omega\times\mR^n\rightarrow\mR^{n\times d}$ are
$\mP\otimes\mathcal {B}(\mR^n)$ measurable, and
$g:[0,T]\times\Omega\times E\times\mR^n\rightarrow\mR^n$ is
$\mP\otimes\mathcal {E}\otimes\mathcal {B}(\mR^n)$ measurable.
The functions $b,\ \sigma$, and $g$ are called drift coefficient,
diffusion coefficient and jump coefficient, respectively. We denote
by $\{X_t^s(x),\ t\in[s,T]\}$ the solution of equation \eqref{b1}
starting from $x$ at time $s$ and simply denote
$X_\cdot(x):=X_\cdot^0(x)$.

We assume the following conditions\\
$\mathbf{(C1)}$ the coefficients $b$, $\sigma$, and $g$ are of
linear growth with respect to $x$, i.e., there exist positive
constant $K$ and deterministic function $K(e)$ such that
$$|b(t,x)|\leq K(1+|x|),\ |\sigma(t,x)|\leq K(1+|x|),\ |g(t,e,x)|\leq
K(e)(1+|x|)$$ and $$\int_EK(e)^pv(de)<\infty,\ \forall p\geq2;$$\\
$\mathbf{(C2)_{k}}$ the coefficients $b$, $\sigma$ and $g$ are
differentiable and their derivatives up to the order $k$ are
bounded, i.e., there exist positive constant $L$ and deterministic
function $L(e)$, such that for $1\leq|\gamma|\leq k$,
$$|\pat^\gamma b(t,x)|\leq L,\ |\pat^\gamma\sigma(t,x)|\leq L,\ |\pat^\gamma g(t,e,x)|\leq L(e)$$
and $$\int_EL(e)^pv(de)<\infty,\ \forall p\geq2;$$ $\mathbf{(C3)}$
the map $\phi_{t,e}:x\rightarrow x+g(t,e,x)$ is homeomorphic for any
$(t,\omega,e)\in[0,T]\times\Omega\times E$ and the inverse map
$\phi^{-1}_{t,e}$ is uniformly Lipschitz continuous and of uniformly
linear growth with respect to $x$;\\
$\mathbf{(C4)}$ the Jacobian matrix $I+\pat g(t,e,x)$ of the
homeomorphic map $\phi_{t,e}(x)$ is invertible for any $x$, for
almost
all $(t,\omega,e)\in[0,T]\times\Omega\times E$.\\
The following two lemmas are due to Kunita \cite{Ku1} or Fujiwara
and Kunita \cite{FuKu}.
\begin{lem}\label{lemb1}
  Assume the conditions $\mathbf{(C1)}$, $\mathbf{(C2)_1}$ and $\mathbf{(C3)}$ are satisfied. Then
  there exists a version of the unique solution of equation
  \eqref{b1}, denoted still by $\{X_t^s(x),\ (t,x)\in[s,T]\times\mR^n\}$, which satisfies\\
  (i) $t\rightarrow X_t^s(\cdot)$ is a $C(\mR^n)$-valued cadlag
  process;\\
  (ii) $X_t^s(\cdot):\mR^n\rightarrow\mR^n$ is homeomophic for any
  $t\in[s,T]$, a.s.;\\
  (iii) $X_r^s(x)=X_r^t(X_t^s(x))$, $0\leq s\leq t \leq r\leq T$.
\end{lem}
\begin{lem}\label{lemb2}
  If conditions $\mathbf{(C1)}$, $\mathbf{(C2)_{k+1}}$, $\mathbf{(C3)}$ and $\mathbf{(C4)}$ are satisfied, the
  unique solution of the equation \eqref{b1} defines a stochastic
  flow of $C^k$-diffeomorphism.
\end{lem}
\section{SIPDE for the inverse flow $X_\cdot^{-1}(x)$}
In this section, under some suitable hypotheses on the coefficients
$b$, $\sigma$ and $g$, we will prove that the $i$-th coordinate of
the inverse flow $X^{-1}_\cdot(x)$ of the solution $X_\cdot(x)$ of
the SDE \eqref{b1} satisfies a stochastic integral partial
differential equation (SIPDE) of the following form
  \begin{equation}
  \left\{
  \begin{split}\label{cc1}
    du(t,x)=&(\mM^r\mM^r-\mL)(t,x)u(t-,x)dt+\int_E\mA(t,e)u(t-,x)v(de)dt\\
    &-\mM^r(t,x)u(t-,x)dW^r_t+\int_E\mA(t,e)u(t-,x)\widetilde{N}(dedt),\
    0\leq t\leq T,\\
    u(0,x)=&x^i
  \end{split}
  \right.
\end{equation}
where the operators $\mL(t,x)$ and $\mM^r(t,x)$ are defined in
\eqref{oper} and
\begin{eqnarray}\label{oper1}
 \mA(t,e)f(x):=-f(x)+f(\phi^{-1}_{t,e}(x)).
\end{eqnarray}
As we will see in Section 5, the above equation plays a crucial role
in the construction of the solution of BSIPDE \eqref{a1}. In fact,
when SDE \eqref{b1} is driven only by a Brownian motion, Krylov and
Rozovskii have proved that the inverse flow $X_\cdot^{-1}(x)$
satisfies equation \eqref{cc1} with $g=0$ (see \cite[Theorem 3.1,
page 89]{KrRo1}). We generalize their results to the case of jump
diffusions.
\subsection{An It\^{o}-Wentzell formula for jump diffusions}
    An It\^{o}-Wentzell formula for forward processes driven by
  a Poisson point process was established by {\O}ksendal and Zhang
  in \cite{OkZh} where the integrands in \eqref{eqc1} are required to be square integrable with respect to $x$ on the whole space
  $\mR^n$, which prevents us from directly applying it to the solution
  $(Y(x),Z(x),U(x))$ of BSDE \eqref{d10}. Recently Krylov \cite{Kr} considered the case of Brownian motion-driven semimartingales and proved an It\^{o}-Wentzell formula for
  distribution-valued processes so that generalized some existing ones (see, for instance, Theorem 3.3.1 of \cite{Ku2} and Theorem 1.4.9 of \cite{Ro}). Our method is essentially same as that of \cite{Kr} where the assumptions, except imposed on the jump coefficients $g$ and $J$, are weaker than
  those in Lemma \ref{lemcc1}. However Lemma \ref{lemcc1} is enough for our subsequent use.

Let $X$ be an $\mathbb{R}^n$-valued stochastic process given by
$$X_t=X_0+\int^t_0b(s)ds+\int^t_0\sigma(s)dW_s+\int_0^t\int_Eg(s,e)\tilde{N}(deds).$$
Here $b(\cdot)$ is predictable $\mR^n$-valued process,
$\sigma(\cdot)$ is predictable $\mR^{n\times d}$-valued process and
$g(\cdot,\cdot)$ is $\mP\otimes\mE$ measurable $\mR^n$-valued
process such that almost surely
\begin{equation}\label{c1}
 \begin{split}
  \int_0^T\big[|b(t)|+\textrm{tr}a(t)\big]dt<\infty,\\
  \sup_{(t,e)\in[0,T]\times E}|g(t,e)|<\infty
  \end{split}
\end{equation}where $2a(t):=\sigma(t)\sigma'(t)$ and $\textrm{tr}a(t):=\sum_{i=1}^{n}|a^{ii}(t)|$.

Let $\{F(t,x),\ (t,x)\in[0,T]\times\mathbb{R}^n\}$ be a family of
$\mathbb{R}$-valued semimartingales of the form
\begin{equation}\label{eqc1}
  F(t,x)=F(0,x)+\int_0^tG(s,x)ds+\int_0^tH(s,x)dW_s+\int_0^t\int_EJ(s,e,x)\widetilde{N}(deds)
\end{equation}
where the $\mR$-valued function $G(\cdot,\cdot)$ and $\mR^d$-valued
function $H(\cdot,\cdot)$ are $\mP\otimes\mB(\mR^n)$ measurable and
$\mR$-valued function $J(\cdot,\cdot,\cdot)$ is
$\mP\otimes\mE\otimes\mB(\mR^n)$
measurable. Assume that\\
$\mathbf{(A1)}$ For any $(\omega,t,e)\in\Omega\times[0,T]\times E$,

(a) the function $F(t,x)$ is twice continuously differentiable in
$x$,

(b) the function $G(t,x)$ is continuous in $x$,

(c) the function $H(t,x)$ is continuously differentiable in $x$,

(d) the function $J(t,e,x)$ is continuous in $x$;\\
$\mathbf{(A2)}$ For any compact subset $K\subset\mR^n$ we have
almost surely
\begin{eqnarray*}
  &&\int_0^T\sup_{x\in K}\big[|F(t,x)|\big(|b(t)|+\textrm{tr}a(t)\big)+|F(t,x)|^2\textrm{tr}a(t)\big]dt<\infty,\label{cc3}\\
  &&\int_0^T\sup_{x\in K}\big[|F(t,x)|^2+|\pat F(t,x)|+|L(t)F(t,x)|+|M(t)F(t,x)|^2\big]<\infty,\label{c2}\\
  &&\int_0^T\sup_{x\in K}\big[|G(t,x)|+|H(t,x)|^2+|M^k(t)H^k(t,x)|
  \big]dt<\infty,\label{c3}\\
  &&\int_0^T\int_E\sup_{x\in K}|J(t,e,x)|^2<\infty,\label{c4}
\end{eqnarray*}
where the differential operators
\begin{equation*}
  \begin{split}
    L(t)&:=a^{ij}(t)\pat^2_{ij}+b^i(t)\pat_i,\\
    M^k(t)&:=\sigma^{ik}(t)\pat_i,\ k=1,\cdots,d,\\
    M(t)&:=(M^1(t),\cdots,M^d(t))'.
  \end{split}
\end{equation*}
\begin{lem}\label{lemcc1}
Suppose that the conditions $\mathbf{(A1)}$ and $\mathbf{(A2)}$ are
satisfied. Then we have for each $t\in[0,T]$ almost surely
\begin{equation}
 \begin{split}\label{c5}
  &F(t,X(t))\\
  =~&F(0,x)+\int_0^tG(s,X_{s-})ds+\int_0^tH(s,X_{s-})dW_s+\int_0^t\langle\pat
  F(s-,X_{s-}),b(s)\rangle ds\\
  &+\int_0^t\langle\pat
  F(s-,X_{s-}),\sigma(s)dW_s\rangle
  +\frac{1}{2}\int_0^t\ll\pat^2F(s-,X_{s-}),\sigma\sigma^*(s)\gg
  ds\\
  &+\int_0^t\ll\pat
  H(s,X_{s-}),\sigma(s)\gg ds+\int_0^t\int_E\big[J(s,e,X_{s-}+g(s,e))-J(s,e,X_{s-})\big]v(de)ds\\
  &+\int_0^t\int_E\big[F(s-,X_{s-}+g(s,e))-F(s-,X_{s-})-\langle\pat
  F(s-,X_{s-}),g(s,e)\rangle\big]v(de)ds\\
  &+\int_0^t\int_E\big[F(s-,X_{s-}+g(s,e))-F(s-,X_{s-})+J(s,e,X_{s-}+g(s,e))\big]\widetilde{N}(deds)
 \end{split}
\end{equation}
where $\ll A,B\gg:=\textrm{tr}(AB')$ for $n\times m$ matrices $A$
and $B$.
\end{lem}
\begin{proof}
 Taking nonnegative $\phi\in C_c^\infty(\mathbb{R}^n,\mathbb{R})$
  with support in the unit ball and
  $\int_{\mR^n}\phi(x)dx=1$, define for $\eps>0$,
  $\phi_\eps(x):=\eps^{-n}\phi(x/\eps)$. Then for any
  $x\in\mR^n$, It\^{o}'s formula yields
  \begin{equation*}
    \begin{split}
    &\phi_\eps(X_t-x)\\
    =~&\phi_\eps(X_0-x)+\int_0^t\langle\pat\phi_\eps(X_{s-}-x),b(s)\rangle
    ds+\int_0^t\la\pat\phi_\eps(X_{s-}-x),\sigma(s)dW_s\ra\\
    &+\frac{1}{2}\int_0^t\ll\pat^2\phi_\eps(X_{s-}-x),\sigma\sigma^*(s)\gg ds\\
    &+\int_0^t\int_E\big[\phi_\eps(X_{s-}-x+g(s,e))-\phi_\eps(X_{s-}-x)\big]\widetilde{N}(deds)\\
    &+\int_0^t\int_E\big[\phi_\eps(X_{s-}-x+g(s,e))-\phi_\eps(X_{s-}-x)-\la\pat\phi_\eps(X_{s-}-x),g(s,e)\ra\big]v(de)ds.
    \end{split}
  \end{equation*}
  Again using It\^{o}'s formula to the product $F(t,x)\phi_\eps(X_t-x)$, we
  obtain almost surely  for all $t\in[0,T]$
  \begin{equation}
    \begin{split}\label{c6}
      &F(t,x)\phi_\eps(X_t-x)\\
      =~&F(0,x)\phi_\eps(X_0-x)+\int_0^t\phi_\eps(X_{s-}-x)G(s,x)ds+\int_0^t\phi_\eps(X_{s-}-x)H(s,x)dW_s\\
      &+\frac{1}{2}\int_0^tF(s-,x)\ll\pat^2\phi_\eps(X_{s-}-x),\sigma\sigma^*(s)\gg
      ds+\int_0^tH(s,x)\sigma^*\pat\phi_\eps(X_{s-}-x)
      ds\\
      &+\int_0^t\int_E\big[\phi_\eps(X_{s-}-x+g(s,e))-\phi_\eps(X_{s-}-x)\big]J(s,e,x)v(de)ds\\
      &+\int_0^t\int_EF(s-,x)\big[\phi_\eps(X_{s-}-x+g(s,e))-\phi_\eps(X_{s-}-x)-\la\pat\phi_\eps(X_{s-}-x),g(s,e)\ra\big]v(de)ds\\
      &+\int_0^tF(s-,x)\la\pat\phi_\eps(X_{s-}-x),\sigma(s)dW_s\ra+\int_0^tF(s-,x)\la\pat\phi_\eps(X_{s-}-x),b(s)\ra
      ds\\
      &+\int_0^t\int_E\big[\phi_\eps(X_{s-}-x+g(s,e))F(s-,x)+\phi_\eps(X_{s-}-x+g(s,e))J(s,e,x)\\
      &-\phi_\eps(X_{s-}-x)F(s-,x)\big]\widetilde{N}(deds).
     \end{split}
  \end{equation}
It is well known that condition \eqref{c1} implies that
$$\sup_{0\leq t\leq T}|X_t|<\infty,\ \ a.s..$$
In view of assumption $\mathbf{(A2)}$, we see that all terms in
\eqref{c6} are almost surely finite. For $r\in\mathbb{N}$, set
$B_r:=\{x\in\mR^n:|x|<r\}$. From \eqref{c1} and $\mathbf{(A2)}$, we
have
\begin{equation}
  \begin{split}\label{cc7}
   &\int_0^T\int_{B_r}|\phi_\eps(X_{s-}-x)H(s,x)|^2+|F(s-,x)\sigma^*(s)\pat\phi_\eps(X_{s-}-x)|^2dxds\\
   \leq&\int_0^T\int_{\mR^n}|\phi_\eps(X_{s-}-x)H(s,x)|^2+|F(s-,x)|^2\textrm{tr}a(s)|\pat\phi_\eps(X_{s-}-x)|^2dxds\\
   \leq&\bigg(\int_0^T\sup_{x\in K(\omega)}|H(s,x)|^2ds\bigg)\int_{\mR^n}|\phi_\eps(x)|^2dx\\
   &+\bigg(\int_0^T\sup_{x\in K(\omega)}|F(s-,x)|^2\textrm{tr}a(s)ds\bigg)\int_{\mR^n}|\pat\phi_\eps(x)|^2dx\\
   <&\infty,\ \ a.s.,
  \end{split}
  \end{equation}
  \begin{equation}
  \begin{split}\label{cc8}
   &\int_0^t\int_E\int_{B_r}\big|\big[\phi_\eps(X_{s-}-x+g(s,e))-\phi_\eps(X_{s-}-x)\big]F(s-,x)\big|^2dxv(de)ds\\
   \leq&\int_0^t\int_E\int_{\mR^n}\big|\big[\phi_\eps(X_{s-}-x+g(s,e))-\phi_\eps(X_{s-}-x)\big]F(s-,x)\big|^2dxv(de)ds\\
   \leq&\bigg(Cv(E)\int_0^T\Big(\sup_{x\in K'(\omega)}|F(s-,x)|^2+\sup_{x\in K(\omega)}|F(s-,x)|^2\Big)ds\bigg)\int_{\mR^n}|\phi_\eps(x)|^2dx\\
   <&\infty,\ \ a.s.,
  \end{split}
\end{equation}
and
\begin{equation}
  \begin{split}\label{cc10}
    &\int_0^t\int_E\int_{B_r}|\phi_\eps(X_{s-}-x+g(s,e))J(s,e,x)|^2dxv(de)ds\\
   \leq&\int_0^t\int_E\int_{\mR^n}|\phi_\eps(X_{s-}-x+g(s,e))J(s,e,x)|^2dxv(de)ds\\
   \leq&\bigg(C\int_0^T\int_E\sup_{x\in K'(\omega)}|J(s,e,x)|^2v(de)ds\bigg)\int_{\mR^n}|\phi_\eps(x)|^2dx\\
   <&\infty,\ \ a.s..
  \end{split}
\end{equation} Here $K(\omega)$ and $K'(\omega)$ are two compact subsets of
$\mR^n$ depending on $\omega$. Integrating with respect to $x$ over
the ball $B_r$ on both sides of \eqref{c6}, using Fubini's Theorem
to interchange $dx$ and $ds$ and the stochastic Fubini's theorem
(see \cite[Theorem 65, pages 208-209]{Pro}) to interchange $dx$ and
$dW_s$, and $dx$ and $d\tilde{N}(deds)$, and then letting
$r\rightarrow\infty$, we obtain
\begin{align}
    \begin{split}\label{c7}
      &\int_{\mR^n}F(t,x)\phi_\eps(X_t-x)dx\\
      =&\int_{\mR^n}F(0,x)\phi_\eps(X_0-x)dx+\int_0^t\int_{\mR^n}\phi_\eps(X_{s-}-x)G(s,x)dxds+\int_0^t\int_{\mR^n}\phi_\eps(X_{s-}-x)H(s,x)dxdW_s\\
      &+\frac{1}{2}\int_0^t\int_{\mR^n}F(s-,x)\ll\pat^2\phi_\eps(X_{s-}-x),\sigma\sigma^*(s)\gg
      dxds+\int_0^t\int_{\mR^n}H(s,x)\sigma^*(s)\pat\phi_\eps(X_{s-}-x)dx
      ds\\
      &+\int_0^t\int_E\int_{\mR^n}F(s-,x)[\phi_\eps(X_{s-}-x+g(s,e))-\phi_\eps(X_{s-}-x)-\la\pat\phi_\eps(X_{s-}-x),g(s,e)\ra]dxv(de)ds\\
      &+\int_0^t\Big\la\int_{\mR^n} F(s-,x)\pat\phi_\eps(X_{s-}-x)dx,\sigma(s)dW_s\Big\ra+\int_0^t\Big\la\int_{\mR^n}F(s-,x)\pat\phi_\eps(X_{s-}-x)dx,b(s)\Big\ra
      ds\\
      &+\int_0^t\int_E\int_{\mR^n}\Big\{\big[\phi_\eps(X_{s-}-x+g(s,e))-\phi_\eps(X_{s-}-x)\big]F(s-,x)\\
      &\ \ \ \ \ \ \ \ \ \ \ \ \ \ \ \ \ +\phi_\eps(X_{s-}-x+g(s,e))J(s,e,x)\Big\}dx\tilde{N}(deds)\\
      &+\int_0^t\int_E\int_{\mR^n}\big[\phi_\eps(X_{s-}-x+g(s,e))-\phi_\eps(X_{s-}-x)\big]J(s,e,x)dxv(de)ds.
    \end{split}
\end{align}
Indeed, noting the inequality \eqref{cc7}, we see from the dominated
convergence theorem that as $r\rightarrow\infty$
$$\int_0^T\bigg|\int_{B_r}\phi_\eps(X_{s-}-x)H(s,x)dx-\int_{\mR^n}\phi_\eps(X_{s-}-x)H(s,x)dx\bigg|^2ds\rightarrow0,\ \textrm{in probability}$$
which implies
$$\int_0^T\int_{B_r}\phi_\eps(X_{s-}-x)H(s,x)dxdW_s\rightarrow\int_0^T\int_{\mR^n}\phi_\eps(X_{s-}-x)H(s,x)dxdW_s,\
\textrm{in probability}.$$ The convergence of other terms can be
proved in a similar manner.

Using integration by parts formula, we have
\begin{eqnarray*}
\begin{split}
\int_{\mR^n}F(s-,x)\ll\pat^2\phi_\eps(X_{s-}-x),\sigma\sigma^*(s)\gg
dx&=\int_{\mR^n}\phi_\eps(X_{s-}-x)\ll\pat^2F(s-,x),\sigma\sigma^*(s)\gg dx,\\
\int_{\mR^n}H(s,x)\sigma^*(s)\pat\phi_\eps(X_{s-}-x)dx&=\int_{\mR^n}\phi_\eps(X_{s-}-x)\ll\pat
H(s,x),\sigma(s)\gg dx,\\
\int_{\mR^n}
F(s-,x)\pat\phi_\eps(X_{s-}-x)dx&=\int_{\mR^n}\phi_\eps(X_{s-}-x)\pat
F(s-,x)dx.
\end{split}
\end{eqnarray*}
Finally, letting $\eps\rightarrow0$ in \eqref{c7}, we can deduce
\eqref{c5} using arguments analogous to the above.
\end{proof}

\subsection{An abstract result}
In this and the next subsections, following Krylov and Rozovskii
\cite{KrRo2}, \cite{KrRo3} and \cite{KrRo1}, we focus on the
derivation of SIPDE \eqref{cc1} for the inverse flow
$X_\cdot^{-1}(x)$. To this end, we first establish an abstract
result.

Let $V$ and $H$ be two separable Hilbert spaces and $V$ is
continuously embedded into $H$ such that $V$ is dense in $H$. The
space $H$ is identified with its dual space $H^*$, consequently
$$V\subset H\cong H^*\subset V^*,$$
where $V^*$ is the dual space of $V$. We denote by the $\|\cdot\|_H$
and $\|\cdot\|_V$ the norms in $H$ and $V$, respectively. Denote by
$(\cdot,\cdot)$ the inner product in $H$ and $\la\cdot,\cdot\ra$ the
duality product between $V$ and $V^*$.

We consider the following abstract form of equation \eqref{cc1}
\begin{equation}
  \left\{
  \begin{split}\label{c8}
    du(t)&=A(t)u(t)dt+\int_E\widetilde{A}(t,e)u(t)v(de)dt+B(t)u(t)dW_t+\int_E\widetilde{A}(t,e)u(t-)\widetilde{N}(dedt),\\
    u_0&\in H
  \end{split}
  \right.
\end{equation}
where the three processes
\begin{equation*}
 \begin{split}
 A(\cdot)&\in\mL^\infty_{\mF}(0,T;\mathscr{L}(V,V^*)),\\
B(\cdot)&\in\mL^\infty_{\mF}(0,T;\mathscr{L}(V,H^d)),\\
\wtil{A}(\cdot,\cdot)&\in\mL^\infty_{\mF}(0,T;\mathscr{L}(H,L^2(\mE,H)))
\end{split}
\end{equation*}satisfy the coercive condition\\
\begin{equation}\label{c9}
\begin{split}
 -2\la
  A(t)u,u\ra+\lambda\|u\|_H^2\geq\alpha\|u\|_V^2+\|B(t)u\|_H^2+\int_E\|\widetilde
  A(t,e)u\|_H^2v(de),\ \forall u\in V,
  \end{split}
\end{equation}
for some $\alpha>0$ and $\lambda\in\mR$.

Existence and uniqueness of solutions to SPDEs driven by a Poisson
random measure or a stable noise are studied by many authors, see
e.g. \cite{AlWuZh}, \cite{Hau1}, \cite{Hau2}, \cite{Myt},
\cite{RoZh}, \cite{Wal}, and references therein. Usually the
operator $A$ is assumed to be the infinitesimal generator of a
strongly continuous semigroup and mild solutions in $H$, rather than
weak solutions (in the PDE sense), are considered. In our setting,
both $A$ and $B$ are random operators and $B$ is a first-order
differential operator, which is a little more complicated than the
case in \cite{RoZh} where $B$ is only Lipschitz continuous from $H$
to $H$ in the diffusion term. We have the following theorem
\begin{thm}\label{thmc1}
  Equation \eqref{c8} has a unique solution $u\in \mL^2_{\mF}(0,T;V)\cap
  \mL^{\infty,2}_{\mF}(0,T;H)$. Moreover,
  \begin{equation}
  \begin{split}\label{cc9}
    \|u(t)\|_H^2=&\|u_0\|_H^2+2\int_0^t\la A(s)u(s),u(s)\ra
    ds+2\int_0^t\int_E(\wtil{A}(s,e)u(s),u(s))v(de)ds\\
    &+2\int_0^t(B(s)u(s),u(s))dW_s+\int_0^t\|B(s)u(s)\|_H^2ds\\
    &+\int_0^t\int_E\|\wtil{A}(s,e)u(s-)\|^2_H+2(u(s-),\wtil{A}(s,e)u(s-))\wtil{N}(deds)\\
    &+\int_0^t\int_E\|\wtil{A}(s,e)u(s-)\|^2_Hv(de)ds.
  \end{split}
\end{equation}
\end{thm}
\begin{proof}
  Since the weak limit of the Galerkin approximation, as noted in \cite[Theorem 1.3]{Par1}, is not
  necessarily an $H$-valued cadlag adapted
  process, we split the proof into two steps.\\
  Step1. For any given $h\in \mL^2_{\mF}(0,T;V)\cap
  \mL^{\infty,2}_{\mF}(0,T;H)$, we first prove the following equation
  \begin{equation}
   \left\{
    \begin{split}\label{c15}
    du(t)&=A(t)u(t)dt+\int_E\widetilde{A}(t,e)u(t)v(de)dt+B(t)u(t)dW_t+\int_E\widetilde{A}(t,e)h(t-)\widetilde{N}(dedt),\\
    u_0&\in H
  \end{split}
  \right.
  \end{equation}
  has a unique solution $u\in\mL^2_{\mF}(0,T;V)\cap
  \mL^{\infty,2}_{\mF}(0,T;H)$. Let $\{\nu_n\}_{n=1}^\infty$ be a basis of $V$
  and a complete orthonomal basis of $H$. For $n\in\mathbb{N}$ and $1\leq i\leq n$, set
  $$V_n:=\textrm{span}\{\nu_1,\nu_2,\cdots,\nu_n\},\   u_{0,n}:=\sum_{i=1}^ng^0_{ni}\nu_i,\ u_n(t):=\sum_{i=1}^ng_{ni}(t)\nu_i$$
  where $g^0_{ni}:=(u_0,\nu_i)$ and
  $g_n(t):=(g_{n1}(t),g_{n2}(t),\cdots,g_{nn}(t))$ is the solution
  of the following It\^{o} equation:
  \begin{equation}
   \left\{
    \begin{split}\label{c10}
      dg_{ni}(t)=&\sum_{j=1}^ng_{nj}(t)\la A(t)\nu_j,\nu_i\ra
      dt+\sum_{j=1}^n\int_Eg_{nj}(t)(\widetilde A(t,e)\nu_j,\nu_i) v(de)dt\\
      &+\sum_{j=1}^ng_{nj}(t)(B(t)\nu_j,\nu_i)
      dW_t+\int_E(\widetilde{A}(t,e)h(t-),\nu_i)\widetilde N(dedt),\\
      g_{ni}(0)=&g^0_{ni}, \  i=1,2,\cdots,n.
    \end{split}
    \right.
  \end{equation}
 It follows from It\^{o}'s formula and condition \eqref{c9} that
 \begin{equation}
   \begin{split}\label{cc11}
     &E\|u_n(t)\|_H^2\\
     =&\|u_n(0)\|_H^2+2E\int_0^t\la A(s)u_n(s),u_n(s)\ra
     ds+2E\int_0^t\int_E(\widetilde A(s,e)u_n(s),u_n(s))v(de)ds\\
     &+\sum_{i=1}^nE\int_0^t|(B(s)u_n(s),\nu_i)|^2ds+\sum_{i=1}^nE\int_0^t\int_E|(\wtil
     {A}(s,e)h(s-),\nu_i)|^2v(de)ds\\
     \leq&\|u_n(0)\|_H^2+E\int_0^t\big[\lambda\|u_n(s)\|_H^2-\alpha\|u_n(s)\|_V^2-\|B(s)u_n(s)\|_H^2-\int_E\|\wtil{A}(s,e)u_n(s)\|_H^2v(de)\big]ds\\
     &E\int_0^t\int_E\|\widetilde
     A(s,e)u_n(s)\|_H^2v(de)ds+v(E)E\int_0^t\|u_n(s)\|_H^2ds+E\int_0^t\|B(s)u_n(s)\|_H^2ds\\
     &+E\int_0^t\int_E\|\wtil
     {A}(s,e)h(s-)\|_H^2v(de)ds.\\
   \end{split}
 \end{equation}
 Gronwall's inequality yields that
 \begin{equation}
   \sup_{0\leq t\leq T}E\|u_n(t)\|_H^2\leq C\Big(1+E\int_0^T\|h(s)\|_H^2ds\Big).
 \end{equation}
 From \eqref{cc11}, we have
 \begin{equation}
   E\int_0^T\|u_n(t)\|_V^2\leq C
 \end{equation}
which implies that there exist a subsequence $\{u_{n_k}\}$ and $u\in
\mL^2_{\mF}(0,T;V)$ such that
 \begin{equation}\label{c11}
   u_{n_k}\rightharpoonup u,\ \textrm{weakly}\ \textrm{in}\ \mL^2_{\mF}(0,T;V).
 \end{equation}
 Let $\{f(t),\ t\in[0,T]\}$ be a
 bounded progressive measurable process on $[0,T]$. It follows from \eqref{c10} that for
 each $\nu_i$ and $k\geq i$,
 \begin{equation*}
 \begin{split}
   &E\int_0^Tf(t)(u_{n_k}(t),\nu_i)dt\\
   =&E\int_0^Tf(t)\bigg[(u_0,\nu_i)+\int_0^t\la
   A(s)u_{n_k}(s),\nu_i\ra ds+\int_0^t\int_E(\widetilde
   A(s,e)u_{n_k}(s),\nu_i)v(de)ds\\
   &+\int_0^t(B(s)u_{n_k}(s),\nu_i)
      dW_s+\int_0^t\int_E(\widetilde{A}(s,e)h(s-),\nu_i)\widetilde
      N(deds)\bigg]dt.
 \end{split}
 \end{equation*}
 Since the operators are bounded, passing
 to the limit in the last equality we get
 \begin{equation}
   \begin{split}\label{c13}
     &E\int_0^Tf(t)(u(t),\nu_i)dt\\
     =&E\int_0^Tf(t)\bigg[(u_0,\nu_i)+\int_0^t\la
   A(s)u(s),\nu_i\ra ds+\int_0^t\int_E(\widetilde
   A(s,e)u(s),\nu_i)v(de)ds\\
   &+\int_0^t(B(s)u(s),\nu_i)
      dW_s+\int_0^t\int_E(\widetilde{A}(s,e)h(s-),\nu_i)\widetilde
      N(deds)\bigg]dt.
   \end{split}
 \end{equation}
 Indeed, it is sufficient to show
 \begin{equation}\label{c12}
   E\int_0^Tf(t)\Big(\int_0^t(B(s)u_{n_k}(s),\nu_i)
      dW_s\Big)dt\rightarrow E\int_0^Tf(t)\Big(\int_0^t(B(s)u(s),\nu_i)
      dW_s\Big)dt,
 \end{equation}
 and the convergence of other terms can be treated in an analogous way.
 Since $$B(\cdot)\in\mL^\infty_{\mF}(0,T;\mathscr{L}(V,H^d)),$$we can
 deduce from \eqref{c11} that for any $t\in[0,T]$,
 \begin{equation*}
   (B(\cdot)u_{n_k}(\cdot),\nu_i)\rightharpoonup (B(\cdot)u(\cdot),\nu_i),\ \textrm{weakly}\
   \textrm{in}\
   \mL^2_{\mF}(0,t;\mR^d).
 \end{equation*}
Since the stochastic integral with respect to a Brownian motion is a
linear and
 strong continuous mapping from $\mL^2_{\mF}(0,t;\mR^d)$ to $L^2(\mathcal
 {F}_t;\mR)$, it is weakly continuous. Therefore,
 \begin{equation*}
   \int_0^t(B(s)u_{n_k}(s),\nu_i)dW_s\rightharpoonup \int_0^t(B(s)u(s),\nu_i)dW_s,\ \textrm{weakly}\ \textrm{in}\
   L^2(\mathcal
 {F}_t;\mR)
 \end{equation*}
 and in particular,
 \begin{equation*}
   E\bigg[f(t)\int_0^t(B(s)u_{n_k}(s),\nu_i)dW_s\bigg]\rightarrow
   E\bigg[f(t)\int_0^t(B(s)u(s),\nu_i)dW_s\bigg].
 \end{equation*}
 Moreover,
 \begin{equation*}
   \bigg|E\Big[f(t)\int_0^t(B(s)u_{n_k}(s),\nu_i)dW_s\Big]\bigg|\leq\frac{1}{2}E|f(t)|^2+CE\int_0^T\|u_{n_k}(s)\|^2ds\leq
   C.
 \end{equation*}
 By the dominated convergence theorem we get \eqref{c12}.
 From \eqref{c13}, it follows that for $a.e.\  (t,\omega)\in[0,T]\times\Omega,$
 \begin{equation}
 \begin{split}\label{c14}
   u(t)=&u_0+\int_0^t
   A(s)u(s)ds+\int_0^t\int_E\widetilde
   A(s,e)u(s)v(de)ds\\
   &+\int_0^tB(s)u(s)
      dW_s+\int_0^t\int_E\widetilde{A}(s,e)h(s-)\widetilde
      N(deds).
 \end{split}
 \end{equation}
By \cite[Theorem 2, page 156]{GyKr}, there exists an $H$-valued
adapted cadlag process $\wtil{u}$ which coincides with $u$ for
$a.e.\ (t,\omega)$ and is equal to the right hand of \eqref{c14} for
all $t\in[0,T]$ a.s.. We identify $\wtil{u}$ with $u$. Furthermore,
we have
\begin{equation}
  \begin{split}
    \|u(t)\|_H^2=&\|u_0\|_H^2+2\int_0^t\la A(s)u(s),u(s)\ra
    ds+2\int_0^t\int_E(\wtil{A}(s,e)u(s),u(s))v(de)ds\\
    &+2\int_0^t(B(s)u(s),u(s))dW_s+\int_0^t\|B(s)u(s)\|_H^2ds\\
    &+\int_0^t\int_E\|\wtil{A}(s,e)h(s-)\|_H^2+2(u(s-),\wtil{A}(s,e)h(s-))\wtil{N}(deds)\\
    &+\int_0^t\int_E\|\wtil{A}(s,e)h(s-)\|_H^2v(de)ds.
  \end{split}
\end{equation}
Applying BDG inequality and condition \eqref{c9}, we have
\begin{equation}
  \begin{split}\label{c16}
    &E\sup_{0\leq t\leq
    T}\|u(t)\|_H^2+\alpha\int_0^T\|u(s)\|_V^2ds\\
    \leq&2\|u_0\|_H^2+2(\lambda+v(E))E\int_0^T\|u(s)\|_H^2ds+4E\sup_{0\leq
    t\leq T}\Big|\int_0^t(B(s)u(s),u(s))dW_s\Big|\\
    &+2E\sup_{0\leq
    t\leq
    T}\Big|\int_0^t\int_E\big[\|\wtil{A}(s,e)h(s-)\|_H^2+2(u(s-),\wtil{A}(s,e)h(s-))\big]\wtil{N}(deds)\Big|\\
    &+2E\int_0^T\int_E\|\wtil{A}(s,e)h(s-)\|_H^2v(de)ds\\
    \leq&2\|u_0\|_H^2+CE\int_0^T\|u(s)\|_H^2ds+CE\int_0^T\int_E\|\wtil{A}(s,e)h(s-)\|_H^2v(de)ds\\
    &+\frac{1}{2}E\sup_{0\leq t\leq
    T}\|u_t\|_H^2
  \end{split}
\end{equation}
where we have used the conclusion (see \cite[page 260-261]{RoZh} for
details)
\begin{equation*}
  \begin{split}
    &E\sup_{0\leq
    t\leq
    T}\Big|\int_0^t\int_E\big[\|\wtil{A}(s,e)h(s-)\|_H^2+2(u(s-),\wtil{A}(s,e)h(s-))\big]\wtil{N}(deds)\Big|\\
    \leq&E\int_0^T\int_E\|\wtil{A}(s,e)h(s-)\|_H^2v(de)ds+\frac{1}{4}E\sup_{0\leq
    t\leq T}\|u(t)\|_H^2.
  \end{split}
\end{equation*}
So
\begin{equation*}
  E\sup_{0\leq t\leq
    T}\|u(t)\|_H^2\leq
    C(\|u_0\|_H^2+E\int_0^T\|u(s)\|_V^2ds+E\int_0^T\|h(s)\|_H^2ds)<\infty
\end{equation*}
which implies $$u\in \mL^2_{\mF}(0,T;V)\cap
  \mL^{\infty,2}_{\mF}(0,T;H).$$
  If $u^1$ and $u^2$ are two solutions of the equation \eqref{c15} in $\mL^2_{\mF}(0,T;V)\cap
  \mL^{\infty,2}_{\mF}(0,T;H)$. By It\^{o}'s formula and condition \eqref{c9}, we have
  \begin{equation}
  \begin{split}
     E\|u^1(t)-u^2(t)\|_H^2+\alpha
     E\int_0^t\|u^1(s)-u^2(s)\|_V^2ds\leq(\lambda+v(E))E\int_0^t\|u^1(s)-u^2(s)\|_H^2ds
  \end{split}
  \end{equation}
  which implies
  $$E\int_0^T\|u^1(s)-u^2(s)\|_V^2ds=0.$$
  By a similar calculation as \eqref{c16}, we have
  $$E\sup_{0\leq t\leq T}\|u^1(t)-u^2(t)\|_H^2=0.$$
Step 2. We use the contraction mapping principle to prove the
existence and uniqueness of the solution of equation \eqref{c8}. Let
$h^1,h^2$ be in $\mL^2_{\mF}(0,t;V)\cap
  \mL^{\infty,2}_{\mF}(0,t;H)$ where $t\in[0,T]$ will be determined later. From Step 1 we know there exist
  $u^1,u^2\in\mL^2_{\mF}(0,t;V)\cap
  \mL^{\infty,2}_{\mF}(0,t;H)$ solving equation \eqref{c15} corresponding to $h^1$ and
  $h^2$ respectively. It follows from It\^{o}'s formula and condition \eqref{c9} that
  \begin{equation*}
    \begin{split}
      &\|u^1(s)-u^2(s)\|_H^2+\alpha\int_0^s\|u^1(r)-u^2(r)\|_V^2dr\\
      \leq&(\lambda+v(E))\int_0^s\|u^1(r)-u^2(r)\|_H^2dr+2\int_0^s\big(B(r)(u^1(r)-u^2(r)),u^1(r)-u^2(r)\big)dW_r\\
      &+\int_0^s\int_E\big[2\big(u^1(r-)-u^2(r-),\wtil{A}(r,e)(h^1(r-)-h^2(r-))\big)\\
      &+\|\wtil{A}(r,e)(h^1(r-)-h^2(r-))\|_H^2\big]\wtil{N}(dedr)\\
      &+\int_0^s\int_E\|\wtil{A}(r,e)(h^1(r-)-h^2(r-))\|_H^2v(de)dr,\ 0\leq s\leq t.\\
   \end{split}
  \end{equation*}
 Gronwall's inequality yields
 \begin{equation}\label{c17}
   E\|u^1(s)-u^2(s)\|_H^2\leq
 e^{(\lambda+v(E))T}E\int_0^t\int_E\|\wtil{A}(r,e)(h^1(r-)-h^2(r-))\|_H^2v(de)dr,\
 0\leq s\leq t.
 \end{equation}
 Using BDG inequality and \eqref{c17} , we can get
 \begin{equation*}
   \begin{split}
     &E\sup_{0\leq s\leq
     t}\|u^1(s)-u^2(s)\|_H^2+E\int_0^t\|u^1(s)-u^2(s)\|_V^2ds\\
     \leq&
     CE\int_0^t\|u^1(s)-u^2(s)\|_H^2ds+CE\int_0^t\int_E\|\wtil{A}(s,e)(h^1(s)-h^2(s))\|_H^2v(de)ds\\
     \leq&CE\int_0^t\|h^1(s)-h^2(s)\|_H^2ds\\
     \leq&CtE\sup_{0\leq s\leq t}\|h^1(s)-h^2(s)\|_H^2\\
     \leq&Ct\bigg(E\sup_{0\leq s\leq
     t}\|h^1(s)-h^2(s)\|_H^2+E\int_0^t\|h^1(s)-h^2(s)\|_V^2ds\bigg).
  \end{split}
 \end{equation*}
Taking $t$ small enough such that $Ct<1$, by contract mapping
theorem we know the equation \eqref{c8} has a unique solution in
$\mL^2_{\mF}(0,t;V)\cap
  \mL^{\infty,2}_{\mF}(0,t;H)$ on the interval $[0,t]$. We repeat the process on intervals
  $[t, 2t]$, $[2t, 3t]$, $\cdots$, and finally obtain the existence and
  uniqueness of the solution of equation \eqref{c8} after finite steps.
  \eqref{cc9} follows from \cite[Theorem 2, page 156]{GyKr}. The proof is complete.
\end{proof}
\subsection{Degenerate case}
In this section, we apply the abstract result proved in the previous
section to our equation \eqref{cc1} and prove that the inverse flow
$X^{-1}_\cdot(x)$ is a classical solution to SIPDE \eqref{cc1}.
Consider the following Cauchy problem
\begin{equation}
  \left\{
  \begin{split}\label{c18}
    du(t,x)=&[a^{ij}(t,x)u_{ij}(t,x)+b^i(t,x)u_i(t,x)+c(t,x)u(t,x)]dt\\
    &+\int_E[-u(t,x)+\rho(t,e,x)u(t,\phi_{t,e}^{-1}(x))]v(de)dt\\
    &+[\wtil{b}^{ik}(t,x)u_i(t,x)+\wtil{c}^k(t,x)u(t,x)]dW^k_t\\
    &+\int_E[-u(t-,x)+\rho(t,e,x)u(t-,\phi_{t,e}^{-1}(x))]\wtil{N}(dedt),\ (t,x)\in[0,T]\times\mR^n,\\
    u(0,x)=&\varphi(x).
  \end{split}
  \right.
\end{equation}
Let $m$ be a nonnegative integer and $K$ be a nonnegative constant.
We make the following three assumptions

I) the coefficients $a^{ij}$, $b^i$, $c$, $\wtil{b}^{ik}$,
$\wtil{c}^{k}$ are predictable for each $x$ and $\rho$ is
$\mP\otimes\mathcal {E}$ measurable for each $x$; the functions
$b^i$, $c$, $\wtil{b}^{ik}$, $\wtil{c}^{k}$, $\rho$ and their
derivatives with respect to $x$ up to the order $m$, and the
function $a^{ij}$ and its derivatives with respect to $x$ up to the
order $m+1$, are bounded by $K$; $g$ and its derivatives up to the
order $m$ are bounded by $K$
  and the determinant of the Jacobian matrix $I+\pat g(t,e,x)$ of the homeomorphic map
  $\phi_{t,e}(x)=x+g(t,e,x)$ is
  bounded below by a positive constant;

II) the matrix
$(a^{ij}-\frac{1}{2}\wtil{b}^{ik}\wtil{b}^{jk})\geq\delta I$, for
some $\delta>0$;

III) $\varphi\in W^m_2$.
\begin{rem}\label{remb1}
  Since $\phi_{t,e}^{-1}(x)=x-g(t,e,\phi_{t,e}^{-1}(x))$, our assumptions on the coefficient $g$ imply that the gradient of $\phi_{t,e}$ and the derivatives of $\phi^{-1}_{t,e}$ up to order $m$ with respect to $x$ are
  bounded.
\end{rem}
\begin{defn}
  A generalized solution of the problem \eqref{c18} is a
  function $u\in\mL^2_{\mF}(0,T;W^1_2)\cap
  \mL^{\infty,2}_{\mF}(0,T;L^2)$ such that for each $\eta\in
  C_0^\infty$ and almost all $(t,\omega)\in[0,T]\times\Omega$,
  \begin{equation}
    \begin{split}\label{c19}
      (u(t),\eta)_0=&(\varphi,\eta)_0+\int_0^t[-(a^{ij}u_i(s),\eta_j)_0+((b^i-a^{ij}_j)u_i(s)+cu(s),\eta)_0]ds\\
      &+\int_0^t\int_E(-u(s)+\rho(s,e)u(s,\phi_{s,e}^{-1}),\eta)_0v(de)ds\\
      &+\int_0^t(\wtil{b}^{ik}u_i(s)+\wtil{c}^ku(s),\eta)_0dW^k_s\\
      &+\int_0^t\int_E(-u(s-)+\rho(s,e)u(s-,\phi_{s,e}^{-1}),\eta)_0\wtil{N}(deds).
    \end{split}
  \end{equation}
\end{defn}
\begin{thm}\label{thmc2}
  Under conditions I), II), and III), the Cauchy problem \eqref{c18} has a unique generalized solution
  $u\in \mL^2_{\mF}(0,T;W^{m+1}_2)\cap \mL^{\infty,2}_{\mF}(0,T;W^m_2)$ such that the relation
  \eqref{c19} holds almost surely for any $\eta\in C_0^\infty$ and
  $t\in[0,T]$. In addition,
  \begin{equation}
  \begin{split}\label{cc19}
    \|u(t)\|_m^2=&\|\varphi\|_m^2+\int_0^t\big[-2(a^{ij}u_i(s),u_j(s))_m+2((b^i-a^{ij}_j)u_i(s)+cu(s),u(s))_m\big]ds\\
    &+2\int_0^t\int_E(-u(s)+\rho(s,e)u(s,\phi_{s,e}^{-1}),u(s))_mv(de)ds\\
    &+2\int_0^t(\wtil{b}^{ik}u_i(s)+\wtil{c}^ku(s),u(s))_mdW_s^k+\int_0^t\sum_{k=1}^d\|\wtil{b}^{ik}u_i(s)+\wtil{c}^ku(s)\|_m^2ds\\
    &+\int_0^t\int_E\big[\|-u(s-)+\rho(s,e)u(s-,\phi^{-1}_{s,e})\|_m^2\\
    &+2(u(s-),-u(s-)+\rho(s,e)u(s-,\phi^{-1}_{s,e}))_m\big]\wtil{N}(deds)\\
    &+\int_0^t\int_E\|-u(s-)+\rho(s,e)u(s-,\phi^{-1}_{s,e})\|_m^2v(de)ds.
  \end{split}
  \end{equation}
\end{thm}
\begin{proof}
To apply the abstract result, we set
\begin{equation}
  V=W^{m+1}_2,H=W^m_2,V^*=W^{m-1}_2.
\end{equation}
For $\zeta\in V$ and $\eta\in V$, it follows from condition I) that
$$|-(a^{ij}\zeta_i(s),\eta_j)_m+((b^i-a^{ij}_j)\zeta_i(s)+c\zeta(s),\eta)_m|\leq
C\|\zeta\|_{m+1}\|\eta\|_{m+1}$$ where $C$ is independent of
$t,\omega,\zeta$ and $\eta$. Consequently, the formula
$$\la
A(t)\zeta,\eta\ra:=-(a^{ij}(t)\zeta_i,\eta_j)_m+((b^i(t)-a^{ij}_j(t))\zeta_i+c(t)\zeta,\eta)_m$$
defines a linear operator
$A(\cdot)\in\mL^\infty_{\mF}(0,T;\mathscr{L}(V,V^*))$ and from the
elementary inequality $2ab\leq\epsilon a^2+\frac{1}{\epsilon}b^2$,
\begin{equation}
  \begin{split}
    2\la
A(t)\eta,\eta\ra=&-2(a^{ij}\eta_i,\eta_j)_m+2((b^i-a^{ij}_j)\eta_i+c\eta,\eta)_m\\
\leq&-2\sum_{|\gamma|=m}(a^{ij}\pat^\gamma\eta_i,\pat^\gamma\eta_j)_0+\epsilon_1\|\eta\|_{m+1}^2+C\|\eta\|_{m}^2.
  \end{split}
\end{equation}
In view of Remark \ref{remb1}, for $\eta\in V$ and $\zeta\in H$ the
formulas
$$B(t)\eta:=(\wtil{b}^{i1}(t)\eta_i+\wtil{c}^1(t)\eta,\cdots,\wtil{b}^{id}(t)\eta_i+\wtil{c}^d(t)\eta)$$
and
$$\wtil{A}(t,e)\zeta:=-\zeta+\rho(t,e)\zeta(\phi^{-1}_{t,e})$$
defines two linear operators
$B(\cdot)\in\mL^\infty_{\mF}(0,T;\mathscr{L}(V,H^d))$ and
$\tilde{A}(\cdot,\cdot)\in\mL^\infty_{\mF}(0,T;\mathscr{L}(H,L^2(\mE,H)))$
respectively. Moreover,
\begin{equation*}
  \begin{split}
    \|B(t)\eta\|_m^2=&(\wtil{b}^{ik}\eta_i+\wtil{c}^k\eta,\wtil{b}^{ik}\eta_i+\wtil{c}^k\eta)_m\\
    =&(\wtil{b}^{ik}\eta_i,\wtil{b}^{ik}\eta_i)_m+2(\wtil{b}^{ik}\eta_i,\wtil{c}^k\eta)_m+(\wtil{c}^k\eta,\wtil{c}^k\eta)_m\\
    \leq&\sum_{|\gamma|=m}(\pat^\gamma(\wtil{b}^{ik}\eta_i),\pat^\gamma(\wtil{b}^{ik}\eta_i))_0+\epsilon\|\eta\|_{m+1}^2+C\|\eta\|_m^2\\
    \leq&\sum_{k=1}^d\sum_{|\gamma|=m}\|\wtil{b}^{ik}\pat^\gamma\eta_i\|_0^2+\epsilon_2\|\eta\|_{m+1}^2+C\|\eta\|_m^2
  \end{split}
\end{equation*}
and
$$\int_E\|\wtil{A}(t,e)\zeta\|_m^2v(de)\leq C\|\zeta\|_m^2.$$
So for any $\eta\in V$,
\begin{equation*}
  \begin{split}
    &2\la
    A(t)\eta,\eta\ra+\|B(t)\eta\|^2_m+\int_E\|\wtil{A}(t,e)\eta\|_m^2v(de)\\
    \leq&-2\sum_{|\gamma|=m}(a^{ij}\pat^\gamma\eta_i,\pat^\gamma\eta_j)_0+(\epsilon_1+\epsilon_2)\|\eta\|_{m+1}^2+C\|\eta\|_m^2+\sum_{|\gamma|=m}\sum_{k=1}^d\|\wtil{b}^{ik}\pat^\gamma\eta_i\|_0^2\\
    \leq&-2\delta\sum_{|\gamma|=m}\sum_{i=1}^n\|\pat^\gamma\eta_i\|^2_0+(\epsilon_1+\epsilon_2)\|\eta\|_{m+1}^2+C\|\eta\|_m^2\\
    \leq&-2\delta\|\eta\|_{m+1}^2+(\epsilon_1+\epsilon_2)\|\eta\|_{m+1}^2+(C+2\delta)\|\eta\|_m^2.\\
  \end{split}
\end{equation*}
Taking $\epsilon_1+\epsilon_2=\delta$, we have
$$-2\la
A(t)\eta,\eta\ra+C\|\eta\|_m^2\geq\delta\|\eta\|_{m+1}^2+\|B(t)\eta\|^2_m+\int_E\|\wtil{A}(t,e)\eta\|_m^2v(de).$$
So the coercive condition \eqref{c9} is satisfied. According to
Theorem \ref{thmc1}, there exists a unique function $u\in
\mL^2_{\mF}(0,T;W^{m+1}_2)\cap\mL^{\infty,2}_{\mF}(0,T;W^m_2)$ such
that almost surely for any $\zeta\in W^{m+1}_2$ and $t\in[0,T]$,
 \begin{equation}
    \begin{split}\label{c20}
      (u(t),\zeta)_m=&(\varphi,\zeta)_m+\int_0^t[-(a^{ij}u_i(s),\zeta_j)_m+((b^i-a^{ij}_j)u_i(s)+cu(s),\zeta)_m]ds\\
      &+\int_0^t\int_E(-u(s)+\rho(s,e)u(s,\phi_{s,e}^{-1}),\zeta)_mv(de)ds\\
      &+\int_0^t(\wtil{b}^{ik}u_i(s)+\wtil{c}^ku(s),\zeta)_mdW^k_s\\
      &+\int_0^t\int_E(-u(s)+\rho(s,e)u(s,\phi_{s,e}^{-1}),\zeta)_m\wtil{N}(deds).
    \end{split}
  \end{equation}
Let $\Delta$ represent the Laplacian on $\mR^n$. It is well known
that the operator $\Lambda:=1-\Delta$ which maps $W^2_2$ into $L^2$
has an inverse $\Lambda^{-1}$ satisfying
$\Lambda^{-1}W^l_2=W^{l+2}_2$ for any integer $l$. Moreover, if $k$
is a nonnegative integer such that $l+k\geq0$, then for $f\in W^k_2$
, $g\in W^{l+k}_2\cap W^k_2$, we have
\begin{equation}\label{c21}
  (\Lambda^{-l}f,g)_{l+k}=(f,g)_k.
\end{equation}
For $\eta\in C_0^\infty$, in view of \eqref{c21}, we can get
\eqref{c19} by replacing $\zeta$ by $\Lambda^{-m}\eta$ in
\eqref{c20}. So $u$ is a generalized solution of the problem
\eqref{c18}. Suppose $\hat{u}\in
\mL^2_{\mF}(0,T;W^{m+1}_2)\cap\mL^{\infty,2}_{\mF}(0,T;W^m_2)$ is
another generalized solution of the problem \eqref{c18}. For
$\zeta\in C_0^\infty$, due to \eqref{c21} again and the fact that
$C_0^\infty$ is dense in $W^{m+1}_2$, we replace $\eta$ by
$\Lambda^m\zeta$ in \eqref{c19} and conclude $\hat{u}$ is also a
solution of the abstract equation \eqref{c8}. Theorem \ref{thmc1}
yields $u=\hat{u}$ so the uniqueness is proved. \eqref{cc19} follows
from \eqref{cc9}. The proof is complete.
\end{proof}
Next we consider the equation \eqref{c18} in the degenerate case,
i.e., the assumption II) is replaced by

II')  the matrix
$(a^{ij}-\frac{1}{2}\wtil{b}^{ik}\wtil{b}^{jk})\geq0$.\\
The following lemma is borrowed from \cite[Remark 2.1, page
340]{KrRo3} with $p=2$.
\begin{lem}\label{lemc1}
  Under conditions I), II') and III), we have for $u\in W^{m+1}_2$,
  \begin{equation}
  \begin{split}\label{cc22}
    &-2(a^{ij}u_i,u_j)_m+2((b^i-a^{ij}_j)u_i+cu,u)_m+\sum_{k=1}^d\|\wtil{b}^{ik}u_i+\wtil{c}^ku\|_m^2\\
    &+2\int_E(-u+\rho(s,e)u(\phi_{s,e}^{-1}),u)_mv(de)+\int_E\|-u+\rho(s,e)u(\phi^{-1}_{s,e})\|_m^2v(de)\\
    \leq&N\|u\|_m^2
  \end{split}
\end{equation}
where the constant $N$ depends only on $K$, $n$, $d$, $m$ and
$v(E)$.
\end{lem}
\begin{thm}\label{thmc3}
Assume that conditions I), II') and III) hold. Then the equation
\eqref{c18} has a unique generalized solution
\begin{equation*}
  u\in
  \mL^{\infty,2}_{\mF}(0,T;W^{m-1}_2)\cap\mL^2_{\mF}(0,T;W^m_2)\cap\mL^{\infty,2}_{\mF,w}(0,T;W^m_2).
\end{equation*}
Moreover,
\begin{equation}\label{cc21}
  E\sup_{t\in[0,T]}\|u\|_m^2\leq
C\|\varphi\|_m^2
\end{equation} where $C$ depends on $v(E)$, $n$, $d$, $K$, $m$ and $T$.
\end{thm}
%
\begin{proof}
Uniqueness. We only need to prove that the equation \eqref{c18} has
solution $u\equiv0$ when $\varphi=0$. Using It\^{o}'s formula (see
\cite[Theorem 2, page 156]{GyKr}) to $\|u(t)\|_0^2$ and
$\|u(t)\|_0^2e^{-Nt}$ where $N$ is the constant in Lemma
\ref{lemc1}, we can get
\begin{equation*}
  \begin{split}
    0\leq &e^{-Nt}\|u(t)\|_0^2\\
    \leq&2\int_0^te^{-Ns}(\wtil{b}^{ik}u_i(s)+\wtil{c}^ku(s),u(s))_0dW_s^k\\
    &+\int_0^t\int_Ee^{-Ns}\big[2(u(s-),-u(s-)+\rho(s,e)u(s-,\phi_{s,e}^{-1}))_0\\
    &+\|-u(s-)+\rho(s,e)u(s-,\phi_{s,e}^{-1})\|_0^2\big]\wtil{N}(deds).
  \end{split}
\end{equation*}
Since a nonnegative local martingale with zero initial value equals
to zero, we have $u(t)=0$, $t\in[0,T]$ almost surely.\\
Existence.
  For $\eps>0$, set $a^{\eps ij}:=a^{ij}+\eps\delta^{ij}$. Denote by
  $u^\eps$ the unique generalized solution of \eqref{c18} with
  $a^{ij}$ replaced by $a^{\eps ij}$.   From Theorem \ref{thmc2}, we
  know
  $u^\eps\in \mL^2_{\mF}(0,T;W^{m+1}_2)\cap\mL^{\infty,2}_{\mF}(0,T;W^m_2)$ and satisfies for
  any $\zeta\in W^{m+1}_2$,
   \begin{equation}
    \begin{split}\label{c22}
      (u^\eps(t),\zeta)_m=&(\varphi,\zeta)_m+\int_0^t[-(a^{\eps ij}u^\eps_i(s),\zeta_j)_m+((b^i-a^{ij}_j)u^\eps_i(s)+cu^\eps(s),\zeta)_m]ds\\
      &+\int_0^t\int_E(-u^\eps(s)+\rho(s,e)u^\eps(s,\phi_{s,e}^{-1}),\zeta)_mv(de)ds\\
      &+\int_0^t(\wtil{b}^{ik}u^\eps_i(s)+\wtil{c}^ku^\eps(s),\zeta)_mdW^k_s\\
      &+\int_0^t\int_E(-u^\eps(s-)+\rho(s,e)u^\eps(s-,\phi_{s,e}^{-1}),\zeta)_m\wtil{N}(deds).
    \end{split}
  \end{equation}
We first prove $u^\eps$ satisfies \eqref{cc21} with $C$ depends only
on $n$, $d$, $K$, $m$, $T$, and $v(E)$. It follows from Lemma
\ref{lemc1}
and \eqref{cc19} that
\begin{equation}
  \begin{split}\label{cc25}
    \|u^\eps(t)\|_m^2\leq&\|\varphi\|_m^2+N\int_0^t\|u^\eps(s)\|_m^2ds+2\int_0^t(\wtil{b}^{ik}u^\eps_i(s)+\wtil{c}^ku^\eps(s),u^\eps(s))_mdW_s^k\\
   &+\int_0^t\int_E\|-u^\eps(s-)+\rho(s,e)u^\eps(s-,\phi^{-1}_{s,e})\|_m^2\\
   &+2(u^\eps(s-),-u^\eps(s-)+\rho(s,e)u^\eps(s-,\phi^{-1}_{s,e}))_m\wtil{N}(deds).
  \end{split}
\end{equation}
From the proof of \cite[Lemma 2.1, page 239]{KrRo3}, we find
\begin{equation}
  \begin{split}\label{cc24}
    |(\wtil{b}^{ik}u^\eps_i(s)+\wtil{c}^ku^\eps(s),u^\eps(s))_m|\leq&\sum_{|\gamma|=m}|(\wtil{b}^{ik}\pat^\gamma u^\eps_i(s),\pat^\gamma u^\eps(s))_0|+C\|u^\eps(s)\|_m^2\\
    =&\sum_{|\gamma|=m}\frac{1}{2}|(\wtil{b}^{ik},\pat^i[(\pat^\gamma u^\eps(s))^2])_0|+C\|u^\eps(s)\|_m^2\\
    =&\sum_{|\gamma|=m}\frac{1}{2}|(-\pat^i\wtil{b}^{ik},(\pat^\gamma u^\eps(s))^2)_0|+C\|u^\eps(s)\|_m^2\\
    \leq&C\|u^\eps(s)\|_m^2.
  \end{split}
\end{equation}
By Gronwall's inequality and BDG inequality, as well as
\eqref{cc24}, we can deduce from \eqref{cc25} that
\begin{equation}
  \begin{split}\label{cc26}
  E\sup_{0\leq t\leq
  T}\|u^\eps(t)\|_m^2\leq C\|\varphi\|_m^2.
  \end{split}
\end{equation}
It follows from the reflexivity of the process space
$\mL^2_{\mF}(0,T;W^m_2)$ that there exist a subsequence, still
denoted by $u^\eps$, and $u\in \mL^2(0,T;W^m_2)$ such that as
$\eps\rightarrow0$,
$$u^{\eps}\rightharpoonup u,\ \textrm{weakly}\ \textrm{in}\ \mL^2_{\mF}(0,T;W^m_2).$$
Next we prove $u^\eps$ is also a Cauchy sequence in the space
$\mL^{\infty,2}(0,T;W^{m-1}_2)$. For any $\eta\in W^m_2$, replacing
$\zeta$ by $\Lambda^{-1}\eta$ in \eqref{c22} and using the relation
\eqref{c21}, we have
  \begin{equation}
    \begin{split}\label{c23}
      (u^\eps(t),\eta)_{m-1}=&(\varphi,\eta)_{m-1}+\int_0^t[-(a^{\eps ij}u^\eps_i(s),\eta_j)_{m-1}+((b^i-a^{ij}_j)u^\eps_i(s)+cu^\eps(s),\eta)_{m-1}]ds\\
      &+\int_0^t\int_E(-u^\eps(s)+\rho(s,e)u^\eps(s,\phi_{s,e}^{-1}),\eta)_{m-1}v(de)ds\\
      &+\int_0^t(\wtil{b}^{ik}u^\eps_i(s)+\wtil{c}^ku^\eps(s),\eta)_{m-1}dW_s\\
      &+\int_0^t\int_E(-u^\eps(s)+\rho(s,e)u^\eps(s,\phi_{s,e}^{-1}),\eta)_{m-1}\wtil{N}(deds).
    \end{split}
  \end{equation}
  The above equality implies that $u^\eps$ is also the unique solution of the
  following evolution equation
  \begin{equation*}
   \left\{
    \begin{split}
    du(t)=&A'(t)u(t)dt+\int_E\widetilde{A}'(t,e)u(t)v(de)dt+B'(t)u(t)dW_t+\int_E\widetilde{A}'(t,e)u(t-)\widetilde{N}(dedt),\\
     u(0)=&\varphi
    \end{split}
    \right.
  \end{equation*}
  with triple
  $$V=W^m_2,\ H=W^{m-1}_2,\ V^*=W^{m-2}_2,$$
  and for $\zeta,\ \eta\in V$, $\psi\in H$,
  \begin{equation*}
  \begin{split}
       \la
A'_t\zeta,\eta\ra&:=-(a^{\eps ij}\zeta_i,\eta_j)_{m-1}+((b^i-a^{ij}_j)\zeta_i+c\zeta,\eta)_{m-1},\\
B'(t)\eta&:=(\wtil{b}^{i1}\eta_i+\wtil{c}^1\eta,\cdots,\wtil{b}^{id}\eta_i+\wtil{c}^d\eta),\\
\wtil{A}'(t,e)\psi&:=-\psi+\rho(t,e)\psi(\phi^{-1}_{t,e}).
  \end{split}
  \end{equation*}
Using It\^{o}'s formula (see again \cite[Theorem 2, page 156]{GyKr})
and Lemma \ref{lemc1}, we have
\begin{equation*}
  \begin{split}
    &\|u^{\eps}(t)-u^{\eps'}(t)\|^2_{m-1}\\
    \leq~&C\int_0^t\|u^{\eps}(s)-u^{\eps'}(s)\|^2_{m-1}ds-2(\eps-\eps')\int_0^t(u^{\eps'}_i(s),u^{\eps}_i(s)-u^{\eps'}_i(s))_{m-1}ds\\
    &+2\int_0^t(\wtil{b}^{ik}(u^{\eps}_i(s)-u^{\eps'}_i(s))+\wtil{c}^k(u^{\eps}(s)-u^{\eps'}(s)),u^{\eps}(s)-u^{\eps'}(s))_{m-1}dW_s^k\\
    &+2\int_0^t\int_E\big[\|-(u^{\eps}(s-)-u^{\eps'}(s-))+\rho(s,e)(u^{\eps}-u^{\eps'})(s-,\phi^{-1}_{s,e})\|_{m-1}^2\\
    &+2(u^{\eps}(s-)-u^{\eps'}(s-),-(u^{\eps}(s-)-u^{\eps'}(s-))+\rho(s,e)(u^{\eps}-u^{\eps'})(s-,\phi^{-1}_{s,e}))_{m-1}\big]\wtil{N}(deds).\\
  \end{split}
\end{equation*}
From \eqref{cc24}, \eqref{cc26} and BDG inequality, we obtain that
\begin{equation}
  \begin{split}
    &E\sup_{0\leq s\leq
    t}\|u^{\eps}(s)-u^{\eps'}(s)\|^2_{m-1}\\
    \leq&~CE\int_0^t\|u^{\eps}(r)-u^{\eps'}(r)\|_{m-1}^2dr+C|\eps-\eps'|E\int_0^t\|u^{\eps'}(r)\|_m\|u^{\eps}(r)-u^{\eps'}(r)\|_mdr\\
    \leq&~C\int_0^tE\sup_{0\leq s\leq
    r}\|u^{\eps}(s)-u^{\eps'}(s)\|^2_{m-1}dr+CT|\eps-\eps'|\|\varphi\|_m^2.
  \end{split}
\end{equation}
Gronwall's inequality yields
$$E\sup_{0\leq s\leq
T}\|u^{\eps}(s)-u^{\eps'}(s)\|^2_{m-1}\leq
C|\eps-\eps'|\|\varphi\|_m^2$$ which implies there exists a
$\tilde{u}\in \mL^{\infty,2}_{\mF}(0,T;W^{m-1}_2)$ such that
$$u^{\eps}\rightarrow\tilde{u},\ \textrm{strongly}\ \textrm{in}\ \mL^{\infty,2}_{\mF}(0,T;W^{m-1}_2).$$
Since $u^{\eps}\rightharpoonup u$ weakly in
$\mL^2_{\mF}(0,T;W^m_2)$, $u^{\eps}\rightharpoonup u$ weakly in
$\mL^2_{\mF}(0,T;W^{m-1}_2)$. So
$u=\tilde{u}\in\mL^{\infty,2}_{\mF}(0,T;W^{m-1}_2)$. Using similar
arguments to deduce \eqref{c13}, we know that both sides of
\eqref{c19}, corresponding to $u^{\eps}$, converge weakly to the
corresponding expression to $u$ in $\mL^2_{\mF}(0,T;\mR)$. So $u\in
\mL^2_{\mF}(0,T;W^m_2)\cap \mL^{\infty,2}_{\mF}(0,T;W^{m-1}_2)$ is a
generalized solution of equation \eqref{c18}.

Noting $u\in\mL^{\infty,2}_{\mF}(0,T;W^{m-1}_2)$, we can prove
$u\in\mL^{\infty,2}_{\mF,w}(0,T;W^m_2)$ and satisfies \eqref{cc21}
in the same way as \cite[Theorem 3.1, page 341]{KrRo3}.
\end{proof}
In order to deduce the SIPDE for the invert flow $X_\cdot^{-1}(x)$,
we need to introduce weighted Sobolev spaces like in \cite{KrRo3} as
the initial value $X^{-1}_0(x)=x$ does not belong to any Sobolev
space on the whole $\mR^n$.

For $p>1$ and $r\in \mR$, we denote by $L_p(r)$ the space of
real-valued Lebesgue measurable functions on $\mR^n$ with the finite
norm
$$\|f\|_{p,r}:=\bigg(\int_{\mR^n}\Big|(1+|x|^2)^{r/2}f(x)\Big|^pdx\bigg)^{1/p}.$$
For $p=2$ we denote the inner product in $L_2(0)$ by
$(\cdot,\cdot)_0$ as before.

Let $W^m_p(r)$ be the subset of $L_p(r)$ consisting of functions
whose generalized derivatives up to the order $m$ belong to
$L_p(r)$. We introduce a norm in this space by
$$\|f\|_{m,p,r}:=\bigg(\sum_{|\gamma|\leq
m}\frac{|\gamma|!}{\gamma^1!\cdots\gamma^n!}\int_{\mR^n}\Big|(1+|x|^2)^{r/2}D^\gamma
f(x)\Big|^pdx\bigg)^{1/p}$$ where
$\gamma=(\gamma^1,\cdots,\gamma^n)$ is a multi-index. It is a Banach
space and for $p=2$ a Hilbert space.

In the remaining part of this section we assume I), II') and

III')  $\varphi\in W^m_2(r)$.
\begin{defn}
  An $r$-generalized solution of the problem \eqref{c18} is a function
  $ u\in \mL^2_{\mF}(0,T;W^1_2(r))\cap \mL^{\infty,2}_{\mF}(0,T;L_2(r))$ such that for each
  $\eta\in C_0^\infty$ and almost all
  $(t,\omega)\in[0,T]\times\Omega$, equation \eqref{c19} holds.
\end{defn}
\begin{rem}
  If an $r$-generalized solution satisfies that for almost all
  $\omega\in\Omega$, it is cadlag with respect to $t$ for any $x$,
  and twice continuously differentiable with respect to $x$ for any
  $t$, it is a classical solution, i.e., it satisfies the relation
  \eqref{c18} for all $(t,x)\in[0,T]\times\mR^n$, almost surely
  for $\omega\in\Omega$.
\end{rem}
\begin{thm}\label{thmc4}
  Assume that conditions I), II') and III') are in force. Then the
  problem
\eqref{c18} has a unique $r$-generalized solution
\begin{equation}\label{cc23}
  u\in\mL^{\infty,2}_{\mF}(0,T;W^{m-1}_2(r))\cap\mL^2_{\mF}(0,T;W^m_2(r))\cap\mL^{\infty,2}_{\mF,w}(0,T;W^m_2(r)).
\end{equation}
Furthermore,
\begin{equation}\label{c24}
  E\sup_{t\in[0,T]}\|u\|_{m,2,r}\leq
C\|\varphi\|_{m,2,r}
\end{equation} where $C$ depends on $v(E)$, $n$, $d$, $K$, $m$, $r$, and
$T$.
\end{thm}
\begin{proof}
  The uniqueness can be obtained by analogous arguments to those of
  Theorem \ref{thmc3}. Now we prove the existence.

  The validity for the case $r=0$ is proved in Theorem \ref{thmc3}.
  For the general case we consider the equation for
  $\wtil{u}(t,x):=(1+|x|^2)^{r/2}u(t,x)$ where $u$ is a solution of
  the problem \eqref{c18}:
\begin{equation}
  \left\{
  \begin{split}
    d\wtil{u}(t,x)=&[a^{ij}(t,x)\wtil{u}_{ij}(t,x)+\overline{{b}^i}(t,x)\wtil{u}_i(t,x)+\overline{c}(t,x)\wtil{u}(t,x)]dt\\
    &+\int_E(-\wtil{u}(t,x)+\overline{d}(t,e,x)\wtil{u}(t,\phi_{t,e}^{-1}(x)))v(de)dt\\
    &+[\wtil{b}^{ik}(t,x)\wtil{u}_i(t,x)+\overline{\wtil{c}^k}(t,x)\wtil{u}(t,x)]dW^k_t\\
    &+\int_E(-\wtil{u}(t-,x)+\overline{d}(t,e,x)\wtil{u}(t-,\phi_{t,e}^{-1}(x)))\wtil{N}(dedt),\ t\in[0,T],\ x\in\mR^n,\\
    \wtil{u}(0,x)=&(1+|x|^2)^{r/2}\varphi(x)
  \end{split}
  \right.
\end{equation}
where
\begin{equation}
  \begin{split}
    &\overline{b^i}(t,x)=b^i(t,x)-2r\frac{a^{ij}(t,x)x_j}{1+|x|^2},\\
   &\overline{c}(t,x)=r(r+2)\frac{a^{ij}(t,x)x_ix_j}{(1+|x|^2)^2}-r\frac{a^{ii}(t,x)+b^i(t,x)x_i}{1+|x|^2}+c(t,x),\\
    &\overline{d}(t,e,x)=(\frac{1+|x|^2}{1+|\phi_{t,e}^{-1}(x)|^2})^{r/2}d(t,e,x),\\
    &\overline{\wtil{c}^k}(t,x)=\wtil{c}^k(t,x)-r\frac{\wtil{b}^{ik}x_i}{1+|x|^2}.\\
  \end{split}
\end{equation}
The coefficients and initial data of the above equation satisfy the
conditions of the theorem for $r=0$. From the validity for the case
$r=0$ we know
\begin{equation}
\begin{split}\label{cc27}
  \wtil{u}\in\mL^{\infty,2}_{\mF}(0,T;W^{m-1}_2)&\cap\mL^2_{\mF}(0,T;W^m_2)\cap\mL^{\infty,2}_{\mF,w}(0,T;W^m_2),\\
  &E\sup_{t\in[0,T]}\|\wtil{u}\|_m\leq C\|\varphi\|_{m,2,r}.
\end{split}
\end{equation}
By differentiating the expression $(1+|x|^2)^{-r/2}\wtil{u}(t,x)$ we
can go back to the equation \eqref{c18} for $u$. It is a consequence
of \eqref{cc27} that $u$ satisfies \eqref{cc23} and \eqref{c24}. The
proof is complete.
\end{proof}
\begin{cor}\label{corc1}
  If $(m-j)2>n$ for some integer $j\geq2$, it follows from the
  Sobolev theorem on the embedding of $W^m_2$ in $C^j$ that the
  $r$-generalized solution of \eqref{c18} is classical.
\end{cor}

\begin{thm}\label{thmc5}
  Suppose the integer $k$ satisfies $k>2+n/2$. Moreover, assume the coefficients of equation \eqref{b1} satisfy: $b$, $\sigma$ and $g$, as well as their derivatives up to the order $k+1$, are bounded;
  the determinant of the Jacobian matrix $I+\nabla g(t,e,x)$ of the homeomorphic map
  $\phi_{t,e}(x)=x+g(t,e,x)$ is bounded below by a positive constant. Then the solution
$X_t(\cdot):\mR^n\rightarrow\mR^n$ of
  equation \eqref{b1} starting at time $0$ is a $C^k$
  diffeomorphism for each $t\in[0,T]$ almost surely. Furthermore,
  the $i$-th
  coordinate of the inverse flow $X_\cdot^{-1}(\cdot)$ is the classical
  solution of the equation \eqref{cc1}.
\end{thm}
\begin{proof}
  \ref{lemb2} implies the first assertion of the theorem. Since $\varphi(x)\equiv x^i\in
  W^m_2(r)$ for any $r\leq-(n/2+1)$, we deduce from Theorem
  \ref{thmc4} and Corollary \ref{corc1}, using Sobolev's embedding theorem that the equation \eqref{cc1}
  has a unique classical solution $u(t,x)$ which satisfies for any
  compact subset $K\subset\mR^n$,
  \begin{equation}\label{cc28}
    E\sup_{0\leq t\leq T}\|u(t,x)\|^2_{C^2(K)}\leq
  C_K\|\varphi\|^2_{m,2,r}.
  \end{equation}
  We apply It\^{o}-Wentzell formula to the expression
  $u(t,X_t(x))$ and obtain that $du(t,X_t(x))\equiv0$. Hence
  $u(t,X_t(x))=x^i$ for any $t\in[0,T]$, a.s.. The proof is
  complete.
\end{proof}
\begin{rem}(see \cite[Remark 2, page359]{Ku1})
 Assume the coefficients $b$, $\sigma$, and $g$ are deterministic
 functions and satisfy the conditions in Theorem \ref{thmc5}. Assume
 further that the diffusion coefficient $\sigma$ is $C^{1,2}$ with
 respect to $(t,x)$. Then the inverse flow $\{X^{-1}_{s,t}(x):=(X^s_t)^{-1}(x),\ 0\leq s\leq t\leq T,\ x\in\mathbb{R}^n\}$
 satisfies the following backward stochastic ordinary differential
 equation
 \begin{equation*}
 \begin{split}
 X^{-1}_{s,t}(x)=&x-\int_s^tb(r,X^{-1}_{r,t}(x))dr-\int_s^t\sigma(r,X^{-1}_{r,t}(x))d\overleftarrow{W_r}-\int_s^t\int_Eg(r,e,\phi_{r,e}^{-1}(X^{-1}_{r,t}(x)))\wtil{N}(ded\overleftarrow{r})\\
 &+2\int_s^tc(r,X^{-1}_{r,t}(x))dr+\int_s^t\int_E\big(g(r,e,X^{-1}_{r,t}(x))-g(r,e,\phi_{r,e}^{-1}(X^{-1}_{r,t}(x)))\big)v(de)dr
 \end{split}
 \end{equation*}
 where $$c(t,x):=\frac{1}{2}\sum_{ij}\frac{\pat\sigma^{,j}(t,x)}{\pat
 x^i}\sigma^{ij}(t,x)$$ and the superscript ``,j'' stands for the
 $j$-th column of the underlying matrix; $d\overleftarrow{W_t}$ and
 $\wtil{N}(ded\overleftarrow{r})$ represent the backward It\^{o}
 integral and backward Poisson integral (see \cite{Ku1} for more
 details).
 \end{rem}
\section{BSDEs with jumps}
In this section, let $p\geq2$ is a given constant. We consider the
following backward stochastic differential equation
\begin{equation}\label{d1}
  Y_t=\xi+\int_t^Tf(s,Y_s,Z_s,U_s)ds-\int_t^TZ_sdW_s-\int_t^T\int_EU_s(e)d\wtil{N}(deds),\
  t\in[0,T]
\end{equation}
under the following conditions

$\mathbf{(H1)}$ the generator
$f:[0,T]\times\Omega\times\mR^l\times\mR^{l\times d}\times
L^2(\mE;\mR^l)\rightarrow\mR^l$ is
$\mP\otimes\mB(\mR^l)\otimes\mB(\mR^{l\times
d})\otimes\mB(L^2(\mE;\mR^l))$ measurable and
$f(\cdot,0,0,0)\in\mL^{2,p}_{\mF}(0,T;\mR^l)$.

$\mathbf{(H2)}$ there exists a constant $L>0$ such that
$$|f(t,y,z,u)-f(t,y',z',u')|\leq L(|y-y'|+|z-z'|+|u-u'|),$$
for all $t\in[0,T]$, $y,y'\in\mR^l$, $z,z'\in\mR^{l\times d}$,
$u,u'\in L^2(\mE;\mR^l)$.

$\mathbf{(H3)}$ $\xi\in L^p(\mF_T;\mR^l)$.
\begin{thm}\label{dthm1}
  Under conditions $\mathbf{(H1)}$, $\mathbf{(H2)}$ and $\mathbf{(H3)}$, BSDE \eqref{d1} has a unique
  solution
  $$(Y,Z,U)\in\mL^{\infty,p}_\mF(0,T;\mR^l)\times\mL^{2,p}_\mF(0,T;\mR^{l\times d})\times\mL^{2,p}_\mF(0,T;L^2(\mE;\mR^l)).$$ Moreover, the solution $(Y,Z,U)$
  satisfies
  \begin{equation}
  \begin{split}\label{d2}
    &\|Y\|_{\mL^{\infty,p}_\mF(0,T;\mR^l)}+\|Z\|_{\mL^{2,p}_\mF(0,T;\mR^{l\times d})}+\|U\|_{\mL^{2,p}_\mF(0,T;L^2(\mE;\mR^l))}\\
    \leq&C(\|f(\cdot,0,0,0)\|_{\mL^{2,p}_{\mF}(0,T;\mR^l)}+\|\xi\|_{L^p})
  \end{split}
  \end{equation}
  for some positive constant $C$ which depends on $T$, $v(E)$, and
  $p$.
\end{thm}
\begin{proof}
  The arguments in \cite[Theorem 5.1 pages 54-56]{KPQ} can be adopted to get the desired existence and uniqueness of solution to equation \eqref{d1}. It remains to prove the estimate
  \eqref{d2}. The techniques are more or less standard now (see \cite{Bri} and \cite{P2}).
  Using It\^{o}'s formula and the Lipschitz continuity of $f$, we get
  \begin{equation}\label{d3}
    \begin{split}
      &|Y_t|^2+\int_t^T|Z_s|^2ds+\int_t^T\int_E|U_s(e)|^2v(de)ds\\
      \leq&|\xi|^2+(2L+\frac{L^2}{\eps_1}+\frac{L^2}{\eps_2}+1)\int_t^T|Y_s|^2ds+\eps_1\int_t^T|Z_s|^2ds\\
      &+\eps_2\int_t^T\int_E|U_s(e)|^2v(de)ds+\int_t^T|f(s,0,0,0)|^2ds-2\int_t^T\la
      Y_s,Z_sdW_s\ra\\
      &-\int_t^T\int_E(2\la
      Y_{s-},U_s(e)\ra+|U_s(e)|^2)\wtil{N}(deds).\\
    \end{split}
  \end{equation}
  For any $0\leq r\leq t$, letting $\eps_1=\eps_2=1$ and taking conditional expectation with
respect to $\mF_r$ on both sides of \eqref{d3}, we have
\begin{equation}\label{d4}
  E[|Y_t|^2\big|\mF_r]\leq
E\big[|\xi|^2+\int_0^T|f(s,0,0,0)|^2ds|\mF_r\big]+C\int_t^TE[|Y_s|^2\big|\mF_r]ds,\
\forall t\in[r,T].
\end{equation}
Using Gronwall's inequality, we have
\begin{equation*}
  E[|Y_t|^2\big|\mF_r]\leq
E\big[|\xi|^2+\int_0^T|f(s,0,0,0)|^2ds\big|\mF_r\big]e^{C(T-t)}.
\end{equation*}
In particular,taking $t=r$, we have
\begin{equation*}
  |Y_r|^2\leq E\big[|\xi|^2+\int_0^T|f(s,0,0,0)|^2ds\big|\mF_r\big]e^{CT}.
\end{equation*}
Using Doob's inequality, we have
\begin{equation}
\begin{split}\label{d5}
  E\Big(\sup_{0\leq r\leq T}|Y_r|^2\Big)^{p/2}&\leq E\bigg(\sup_{0\leq r\leq
T}\Big(E\big[|\xi|^2+\int_0^T|f(s,0,0,0)|^2ds\big|\mF_r\big]e^{CT}\Big)^{p/2}\bigg)\\
&\leq
CE\bigg(|\xi|^p+\Big(\int_0^T|f(s,0,0,0)|^2ds\Big)^{p/2}\bigg).
\end{split}
\end{equation}
Taking $\eps_1=\frac{1}{2}$ and $t=0$ in \eqref{d3}, we have
\begin{equation}
  \begin{split}\label{d6}
   &E\Big(\int_0^T|Z_s|^2ds\Big)^{p/2}+E\bigg(\int_0^T\int_E|U_s(e)|^2N(deds)\bigg)^{p/2}\\
  \leq &C\bigg\{E|\xi|^p+E\sup_{0\leq t\leq
T}|Y_t|^p+\eps_2^{p/2}E\bigg(\int_0^T\int_E|U_s(e)|^2v(de)ds\bigg)^{p/2}\\
&+E\bigg(\int_0^T|f(s,0,0,0)|^2ds\bigg)^{p/2}+E\Big|\int_0^T\la
      Y_s,Z_sdW_s\ra\Big|^{p/2}\\
      &+E\Big|\int_0^T\int_E\la
      Y_{s-},U_s(e)\ra\wtil{N}(deds)\Big|^{p/2}\bigg\}.
  \end{split}
\end{equation}
Using BDG inequality, we have
\begin{equation}
  \begin{split}\label{d7}
  E\Big|\int_0^T\la Y_s,Z_sdW_s\ra\Big|^{p/2}&\leq
CE\Big(\int_0^T|Y_s|^2|Z_s|^2ds\Big)^{p/4}\\
&\leq CE\bigg(\sup_{0\leq t\leq T}|Y_t|^{p/2}\Big(\int_0^T|Z_s|^2ds\Big)^{p/4}\bigg)\\
&\leq\frac{C^2}{2\eps}E\sup_{0\leq t\leq
T}|Y_t|^p+\frac{\eps}{2}E\Big(\int_0^T|Z_s|^2ds\Big)^{p/2}
  \end{split}
\end{equation}
and
\begin{equation}
  \begin{split}\label{d8}
  &E\Big|\int_0^T\int_E\la
      Y_{s-},U_s(e)\ra\wtil{N}(deds)\Big|^{p/2}\\
      \leq~&
CE\Big(\int_0^T\int_E|Y_{s-}|^2|U_s(e)|^2N(deds)\Big)^{p/4}\\
=~&CE\Big(\sum_{0<s\leq T}|Y_{s-}|^2|U_s(\mathbbm{p}(s))|^2\Big)^{p/4}\\
\leq~&CE\bigg(\sup_{0\leq s\leq
T}|Y_s|^{p/2}\Big(\int_0^T\int_E|U_s(e)|^2N(deds)\Big)^{p/4}\bigg)\\
\leq~& \frac{C^2}{2\eps}E\sup_{0\leq s\leq
T}|Y_s|^p+\frac{\eps}{2}E\Big(\int_0^T\int_E|U_s(e)|^2N(deds)\Big)^{p/2}.
\end{split}
\end{equation}
Moreover,
\begin{equation}
  \begin{split}\label{d9}
  &E\bigg(\int_0^T\int_E|U_s(e)|^2v(de)ds\bigg)^{p/2}\\
  \leq&~(v(E)T)^{p/2-1}E\int_0^T\int_E|U_s(e)|^pv(de)ds\\
  \leq&~(v(E)T)^{p/2-1}E\int_0^T\int_E|U_s(e)|^pN(deds)\\
  \leq&~(v(E)T)^{p/2-1}E\sum_{0<s\leq T}|U_s(\mathbbm{p}(s))|^p\\
  =&~(v(E)T)^{p/2-1}E\sum_{0<s\leq T}\big(|U_s(\mathbbm{p}(s))|^2\big)^{p/2}\\
  \leq&~(v(E)T)^{p/2-1}E\Big(\sum_{0<s\leq T}|U_s(\mathbbm{p}(s))|^2\Big)^{p/2}\\
  =&~(v(E)T)^{p/2-1}E\bigg(\int_0^T\int_E|U_s(e)|^2N(deds)\bigg)^{p/2}.
\end{split}
\end{equation}
Taking $\eps_2$ in \eqref{d6} and $\eps$ in \eqref{d7} and
\eqref{d8} small enough, we can deduce from \eqref{d6} that
\begin{equation*}
  \begin{split}
     &E\bigg(\int_0^T|Z_s|^2ds\bigg)^{p/2}+E\bigg(\int_0^T\int_E|U_s(e)|^2N(deds)\bigg)^{p/2}\\
  \leq &C\bigg\{E|\xi|^p+E\sup_{0\leq t\leq
T}|Y_t|^p+E\bigg(\int_0^T|f(s,0,0,0)|^2ds\bigg)^{p/2}\bigg\}.
  \end{split}
\end{equation*}
Combining the estimates \eqref{d5} and \eqref{d9}, we can get the
desired result.
\end{proof}
Now we consider the following BSDE
\begin{equation}\label{d10}
 \left\{
 \begin{array}{ccccc}
 \begin{split}
 dY_t&=-f(t,X_t(x),Y_t,Z_t,U_t)dt+Z_tdW_t+\int_EU_t(e)\wtil{N}(dedt),\
 t\in[0,T],\\
 Y_T&=\varphi(X_T(x))
 \end{split}
\end{array}
\right.
\end{equation} where $X$ is the solution of equation
\eqref{b1} starting at time 0 and coefficients $f$ and $\varphi$
satisfy the following conditions
\begin{equation*}
  \begin{split}
  \mathbf{(C5)_k}
\ \ \ \ \ \ \ \ \ \ \ \ \ \ \
&f\in\mL_\mF^\infty(0,T;C^k(\mR^n\times\mR^l\times\mR^{l\times
d}\times L^2(\mE;\mR^l);\mR^l)),\\
&\varphi\in L^\infty(\mF_T;C^k(\mR^n;\mR^l)).
\end{split}
\end{equation*}
In what follows the derivatives of $f$ with respect to $u$ are
Fr\`{e}chet derivatives.

The adapted solution to BSDE \eqref{d10} will be denoted by
$(Y(x),Z(x),U(x))$. We have the following lemma.
\begin{lem}(\cite[Proposition 3.3]{BuPa})\label{dlem1}
  Assume $\mathbf{(C1)}$, $\mathbf{(C2)_k}$ and $\mathbf{(C5)_k}$ are satisfied for $k=1$ or $k=2$.
  Then the unique adapted solution $(Y(x),Z(x),U(x))$ to BSDE
  \eqref{d10} satisfies almost surely
  \begin{equation*}
  \begin{split}
  Y&\in
C^{k-1}(\mR^n;D(0,T;\mR^l)),\\
Z&\in C^{k-1}(\mR^n;L^2(0,T;\mR^{l\times d})),\\
U&\in C^{k-1}(\mR^n;L^2(0,T;L^2(\mE;\mR^l))).
\end{split}
\end{equation*}
And for any
$p\geq2$ and any multi-index $\gamma$ with $|\gamma|\leq k-1$, we
have
\begin{equation*}\label{d11}
  \begin{split}
    \|\pat^\gamma
    Y_\cdot(x)\|^p_{\mL^{\infty,p}_\mF(0,T;\mR^l)}+\|\pat^\gamma
    Z_\cdot(x)\|^p_{\mL^{2,p}_\mF(0,T;\mR^{l\times d})}
    +\|\pat^\gamma U_\cdot(x)\|^p_{\mL^{2,p}_\mF(0,T;L^2(\mE;\mR^l))}
    \leq
    C_p,\ \ \forall x\in\mR^n,
  \end{split}
\end{equation*} and
\begin{equation*}\label{d12}
  \begin{split}
    &\|\pat^\gamma Y_\cdot(x_1)-\pat^\gamma
    Y_\cdot(x_2)\|^p_{\mL^{\infty,p}_\mF(0,T;\mR^l)}+\|\pat^\gamma
    Z_\cdot(x_1)-\pat^\gamma
    Z_\cdot(x_2)\|^p_{\mL^{2,p}_\mF(0,T;\mR^{l\times d})}\\
    &+\|\pat^\gamma U_\cdot(x_1)-\pat^\gamma
    U_\cdot(x_2)\|^p_{\mL^{2,p}_\mF(0,T;L^2(E;\mR^l))}\\
    \leq&
    C_p|x_1-x_2|^p,\ \ \forall x_1,x_2\in\mR^n.
  \end{split}
\end{equation*}
Moreover, the gradient $(\pat Y,\pat Z,\pat U)$ satisfies the
following BSDE
\begin{equation}
\begin{split}
  \pat Y_t(y)=&\pat\varphi(X_T(y))\pat
  X_T(y)+\int_t^T[f_x(\Xi^y_s)\pat X_s(y)+f_y(\Xi^y_s)\pat
  Y_s(y)+f_z(\Xi^y_s)\pat Z_s(y)\\
  &+f_u(\Xi^y_s)\pat
  U_s(y)]ds-\int_t^T\pat Z_s(y)dW_s-\int_t^T\int_E\pat
  U_s(e,x)\wtil{N}(deds)
\end{split}
\end{equation}
where $\Xi^y_s:=(s,X_s(y),Y_s(y),Z_s(y),U_s(y))$ and $\pat
Z_t(y)dW_t:=\pat Z_t^{~,r}(y)dW_t^r$.
\end{lem}
To get higher regularity of $(Y(x),Z(x),U(x))$ with respect to $x$,
we introduce the following assumption

$\mathbf{(C6)}$ the function $f(t,x,y,z,u)$ is linear in $z$ and $u$
with the derivatives $f_z$ and $f_u$ being bounded and being
independent of $(x,y,z,u)$.
\begin{thm}\label{dthm2}
Suppose that for some positive integer $k$, $\mathbf{(C1)}$,
$\mathbf{(C2)_k}$, $\mathbf{(C5)_k}$ and $\mathbf{(C6)}$ are all
satisfied. Then we have almost surely
\begin{equation*}
\begin{split}
  Y&\in C^{k-1}(\mR^n;D(0,T;\mR^l)),\\
  Z&\in
C^{k-1}(\mR^n;L^2(0,T;\mR^{l\times d})),\\
U&\in C^{k-1}(\mR^n;L^2(0,T;L^2(\mE;\mR^l))).
\end{split}
\end{equation*} And for any
$p\geq2$ and any multi-index $\gamma$ with $|\gamma|\leq k-1$, we
have
\begin{equation}\label{d13}
  \begin{split}
    &\|\pat^\gamma
    Y_\cdot(x)\|^p_{\mL^{\infty,p}_\mF(0,T;\mR^l)}+\|\pat^\gamma
    Z_\cdot(x)\|^p_{\mL^{2,p}_\mF(0,T;\mR^{l\times d})}
    +\|\pat^\gamma
    U_\cdot(x)\|^p_{\mL^{2,p}_\mF(0,T;L^2(\mE;\mR^l))}\\
    \leq&
    C_p,\ \ \forall x\in\mR^n,
  \end{split}
\end{equation}
 and
\begin{equation}\label{d14}
  \begin{split}
    &\|\pat^\gamma Y_\cdot(x_1)-\pat^\gamma
    Y_\cdot(x_2)\|^p_{\mL^{\infty,p}_\mF(0,T;\mR^l)}+\|\pat^\gamma
    Z_\cdot(x_1)-\pat^\gamma
    Z_\cdot(x_2)\|^p_{\mL^{2,p}_\mF(0,T;\mR^{l\times d})}\\
    &+\|U_\cdot(x_1)-U_\cdot(x_2)\|^p_{\mL^{2,p}_\mF(0,T;L^2(\mE;\mR^l))}\\
    \leq&
    C_p|x_1-x_2|^p,\ \ \forall x_1,x_2\in\mR^n.
  \end{split}
\end{equation}
Moreover, the triple $(\pat^\gamma Y(x),\pat^\gamma Z(x),\pat^\gamma
U(x))$ satisfies the following BSDE
\begin{equation}
\left\{
\begin{array}{cc}\label{d15}
    \begin{split}
      \pat^\gamma
    Y_t(x)=&-\{f_y(t,X_t(x),Y_t(x))\pat^\gamma Y_t(x)+f_z(t)\pat^\gamma
    Z_t(x)+f_u(t)\pat^\gamma
    U_t(x)\}dt\\
    &-\{f_x(t,X_t(x),Y_t(x))\pat^\gamma
    X_t(x)+P_\gamma(t,x)\}dt+\pat^\gamma Z_t(x)dW_t\\
    &+\int_E\pat^\gamma
    U_t(e,x)\wtil{N}(dedt)\\
    \pat^\gamma Y_T(x)=&\pat^\gamma[\varphi(X_T(x))]
    \end{split}
    \end{array}
    \right.
\end{equation}
where $P_\gamma(t,x)$ is a $n$-dimensional vector whose components
are polynomials of the partial derivatives up to order $|\gamma|-1$
of the components of $X_t(x)$ and $Y_t(x)$, with the partial
derivatives of order $|\gamma|$ of the components of $f$ as
coefficients.
\end{thm}
\begin{proof}
  We use the principle of induction. For the case $k=2$, Theorem
  \ref{dthm2} holds in view of Lemma \ref{dlem1}. Suppose that
  Theorem \ref{dthm2} is true for $k>2$. Now we prove Theorem
  \ref{dthm2} is true for the case $k+1$.
  By Theorem \ref{dthm1} and
  the induction assumptions, we can proceed as \cite[Proposition3.3]{BuPa} to deduce from the equation
  \eqref{d15} that for any $p\geq2$ and multi-index $\gamma$ with $|\gamma|=k-1$, there exists a positive constant $C_p$ such that
  for any $x,x'\in\mR^n$, $h,h'\in\mR$, $1\leq i\leq n$,
  \begin{equation*}
    \begin{split}
      &\|\Delta^i_h\pat^\gamma Y(x)\|^p_{\mL^{\infty,p}_\mF(0,T;\mR^l)}+\|\Delta^i_h\pat^\gamma Z(x)\|^p_{\mL^{2,p}_\mF(0,T;\mR^{l\times d})}\\
      +&\|\Delta^i_h\pat^\gamma U(x)\|^p_{\mL^{2,p}_\mF(0,T;L^2(\mE;\mR^l))}\\
      \leq& C_p,\\
      &\|\Delta^i_h\pat^\gamma Y(x)-\Delta^i_{h'}\pat^\gamma Y(x')\|^p_{\mL^{\infty,p}_\mF(0,T;\mR^l)}+\|\Delta^i_h\pat^\gamma Z(x)-\Delta^i_{h'}\pat^\gamma Z(x')\|^p_{\mL^{2,p}_\mF(0,T;\mR^{l\times d})}\\
      +&\|\Delta^i_h\pat^\gamma U(x)-\Delta^i_{h'}\pat^\gamma U(x')\|^p_{\mL^{2,p}_\mF(0,T;L^2(\mE;\mR^l))}\\
      \leq&
      C_p(|x-x'|^p+|h-h'|^p)
    \end{split}
  \end{equation*}
  where
  $\Delta^i_h\Phi(x):=\frac{1}{h}(\Phi(x+he_i)-\Phi(x))$
  and $(e_1,\cdots,e_n)$ is the orthogonal basis of $\mR^n$. From Kolmogorov's theorem (see, for instance, \cite[Theorem 1.4.1, page 31]{Ku2}), there exists a modification, still denoted by $(Y(x),Z(x),U(x))$, of the solution
  of equation \eqref{d10} such that almost surely
  \begin{equation*}
  \begin{split}
  Y&\in
C^{k}(\mR^n;D(0,T;\mR^l)),\\
 Z&\in C^{k}(\mR^n;L^2(0,T;\mR^{l\times
d})),\\
U&\in C^{k}(\mR^n;L^2(0,T;L^2(\mE;\mR^l))).
\end{split}
\end{equation*}
And the estimates \eqref{d13} and
\eqref{d14} hold for $|\gamma|= k$. In view of the equations
corresponding to $(\pat^\gamma Y(x+he_i),\pat^\gamma
Z(x+he_i),\pat^\gamma U(x+he_i))$ and $(\pat^\gamma Y(x),\pat^\gamma
Z(x),\pat^\gamma U(x))$ and passing to the limit as $h\rightarrow0$
in the expression
$$\frac{1}{h}(\pat^\gamma Y(x+he_i)-\pat^\gamma Y(x)),$$ we can
obtain that for $|\gamma|=k$, $(\pat^\gamma Y(x),\pat^\gamma
Z(x),\pat^\gamma U(x))$ satisfies some BSDE of the form \eqref{d15}.
The proof is complete.
\end{proof}
To get the differentiability of $Z_t(x)$ with respect to $x$ for
$(t,\omega)\in[0,T]\times\Omega$ and the differentiability of
$U_t(e,x)$ with respect to $x$ for $(t,e,\omega)\in[0,T]\times
E\times\Omega$, we need the following version of Kolmogorov's
continuity criterion.
\begin{lem}(\cite[Lemma 1, pages 46-47]{Sz})\label{dlem2}
  Let $\mathbb{B}$ be a Banach space and $X\in
  C^{k,a}(\mR^n;L^p(\mF;\mathbb{B}))$ with $k+a>\frac{n}{p}+j$
  ($j$ is a nonnegative integer). Then there is unique mapping
  $\wtil{X}:\mR^n\times\Omega\rightarrow\mathbb{B}$ such that

  (i) for each $\omega\in\Omega$, $\wtil{X}(\cdot,\omega)\in
  C^j(\mR^n;\mathbb{B})$;

  (ii) for each $x\in\mR^n$, $\wtil{X}(x,\cdot)$ is $\mF$ measurable, and the derivatives $\pat^\gamma\wtil{X}(x,\omega)$ for all
 multi- index $\gamma$ such that $0\leq|\gamma|\leq j$ are
 indistinguishable with the derivatives $\pat^\gamma X$ in
 $L^p(\mF;\mathbb{B})$.

 Moreover, if $B:=B(x_0,R)$ is a disk of $\mR^n$, there is a
 nonnegative constant $C_{k,a,p,B}$ such that for all $X\in
  C^{k,a}(\mR^n;L^p(\mF;\mathbb{B}))$ we have
  $$\|\|\wtil{X}(\cdot,\cdot)\|_{C^j(B;\mathbb{B})}\|_{L^p}\leq
  C_{k,a,p,B}\|X\|_{C^{k,a}(B;L^p)}$$
  where $\wtil{X}$ is a regular version of $X$ satisfying (i) and
  (ii).
\end{lem}
As a consequence of Theorem \ref{dthm2} and Lemma \ref{dlem2}, we
have
\begin{thm}\label{dthm3}
  Suppose $\mathbf{(C1)}$, $\mathbf{(C2)_{k+1}}$ and $\mathbf{(C5)_{k+1}}$ are satisfied with
  $k>\frac{n}{2}+j$ for some positive integer $j\geq2$, and $\mathbf{(C6)}$
  is satisfied. Let $(Y(x),Z(x),U(x))$ be the solution of BSDE \eqref{d10}. We have for any compact subset $K\subseteq\mR^n$,
  \begin{equation*}
    \begin{split}
      &E\sup_{0\leq t\leq T}\|Y_t(\cdot)\|_{C^k(K;\mR^l)}<\infty,\\
      &E\int_0^T\|Z_t(\cdot)\|^2_{C^j(K;\mR^{l\times
  d})}dt<\infty,\\
  &E\int_0^T\int_E\|U_t(e,\cdot)\|^2_{C^j(K;\mR^l)}v(de)dt<\infty.
    \end{split}
  \end{equation*}
\end{thm}
\section{Classical solutions to BSIPDEs}
In this section, we consider classical solutions to BSIPDEs driven
by a Brownian motion and a Poisson point process. First of all, we
establish the relationship between BSIPDEs and non-Markovian FBSDEs.

Denote by $(X^s_t(x),Y^s_t(x),Z^s_t(x),U^s_t(\cdot,x))_{t\in[s,T]}$
the solution of the following non-Markovian FBSDE
\begin{equation}
\left\{
\begin{split}\label{4ee1}
dX_t&=b(t,X_t)dt+\sigma(t,X_t)dW_t+\int_Eg(t,e,X_{t-})\wtil{N}(dedt),\\
dY_t&=-f(t,X_t,Y_t,Z_t,U_t)dt+Z_tdW_t+\int_EU_t(e)\wtil{N}(dedt),\\
X_s&=x,\ Y_T=\varphi(X_T),\ t\in[s,T].
\end{split}
\right.
\end{equation}
\begin{thm}
  Suppose $p(s,x):=Y^s_s(x)$ can be written as a semimartingale of the following form
    \begin{equation}\label{4ee2}
  p(s,x)=\varphi(x)-\int_s^T\Phi(l,x)dl-\int_s^Tq(l,x)dW_l-\int_s^T\int_Er(l,e,x)\wtil{N}(dedl),\
  s\in[0,T].
  \end{equation}
  Then $(p,q,r)$ formally satisfies the BSIPDE \eqref{a1}.
\end{thm}
\begin{proof}
  Applying It\^{o}-Wentzell formula to $p(l,X^s_l(x))$, we have
  \begin{equation}
 \begin{split}\label{4ee3}
  &p(t,X^s_t(x))\\
  =&p(s,x)+\int_s^t\Phi(l,X_{l-}^s(x))dl+\int_s^tq(l,X_{l-}^s(x))dW_l+\int_s^t\langle\pat
  p(l-,X_{l-}^s(x)),b(l,X_{l}^s(x))\rangle dl\\
  &+\int_s^t\langle\pat
  p(l-,X_{l-}^s(x)),\sigma(l,X_{l}^s(x))dW_l\rangle
  +\frac{1}{2}\int_s^t\ll\pat^2p(l-,X_{l-}^s(x)),\sigma\sigma^*(l,X_{l-}^s(x))\gg
  dl\\
  &+\int_s^t\ll\pat
  q(l,X_{l-}^s(x)),\sigma(l,X_{l-}^s(x))\gg dl\\
  &-\int_s^t\int_E\big[r(l,e,\phi_{l,e}(X_{l-}^s(x)))-r(l,e,X_{l-}^s(x))\big]v(de)dl\\
  &+\int_s^t\int_E\big[p(l-,\phi_{l,e}(X_{l-}^s(x)))-p(l-,X_{l-}^s(x))-\langle\pat
  p(l-,X_{l-}^s(x)),g(l,e,X_{l-}^s(x))\rangle\big]v(de)dl\\
  &+\int_s^t\int_E\big[p(l-,\phi_{l,e}(X_{l-}^s(x)))-p(l-,X_{l-}^s(x))+r(l,e,\phi_{l,e}(X_{l-}^s(x)))\big]\widetilde{N}(dedl).
 \end{split}
\end{equation}
On the other hand,
\begin{equation}\label{4ee4}
\begin{split}
p(t,X^s_t(x))=&Y^t_t(X^s_t(x))\\
=&Y^s_t(x)\\
=&p(s,x)-\int_s^tf(l,X_l^s(x),Y_l^s(x),Z_l^s(x),U_l^s(e,x))dl\\
&+\int_s^tZ_l^s(x)dW_l+\int_s^t\int_EU^s_l(e,x)\wtil{N}(dedl).
\end{split}
\end{equation}
Comparing \eqref{4ee3} with \eqref{4ee4}, we obtain
\begin{equation}
  \begin{split}\label{4ee5}
  &\Phi(s,x)\\
  =&-f(s,x,p(s,x),q(s,x)+\pat
  p(s-,x)\sigma(s,x),p(s-,\phi_{s,\cdot}(x))-p(s-,x)+r(s,\cdot,\phi_{s,\cdot}(x)))\\
&-\langle\pat
  p(s-,x),b(s,x)\rangle-\frac{1}{2}\ll\pat^2p(s-,x),\sigma\sigma^*(s,x)\gg\\
  &-\ll\pat
  q(s,x),\sigma(s,x)\gg-\int_E\big[r(s,e,\phi_{s,e}(x))-r(s,e,x)\big]v(de)\\
  &-\int_E[p(s-,\phi_{s,e}(x))-p(s-,x)-\langle\pat
  p(s-,x),g(s,e,x)\rangle\big]v(de).
  \end{split}
\end{equation}
Taking \eqref{4ee5} into \eqref{4ee2}, we can get the desired
result.
\end{proof}
Next we concern the construction of the classical solution of BSIPDE
\eqref{a1} via the solution of FBSDE \eqref{4ee1}.
\begin{defn}\label{edef1}
  A triple of random fields
  $\{(p(t,x),q(t,x),r(t,e,x)),(t,x,e)\in[0,T]\times\mR^n\times E\}$
  is called an adapted classical solution of BSIPDE \eqref{a1}, if

    (i) $p(\cdot,x)$ is an adapted cadlag process for any
    $x$ and is twice continuously differentiable with respect
    to $x$ for any $t$ almost surely;

    (ii) $q(\cdot,x)$ is a predictable process for any $x$ and is continuously differentiable with respect
    to $x$ for almost any $(t,\omega)\in[0,T]\times\Omega$;

    (iii) $r(\cdot,\cdot,x)$ is $\mathcal {P}\otimes\mathcal {E}$ measurable for any $x$ and is continuous with respect to $x$ for almost
    all $(t,\omega,e)\in[0,T]\times\Omega\times E$;

    (iv) for any compact subset $K\subset\mR^n$ and multi-indices $\beta$ and
  $\gamma$ with $|\beta|\leq2$ and $|\gamma|\leq1$, the triple $(p,q,r)$ satisfies almost
  surely
  \begin{equation}
  \begin{split}\label{e2}
    \sup_{0\leq t\leq T,\ x\in K}|\pat^\beta p(t,x)|<\infty,\\
    \int_0^T\sup_{x\in K}|\pat^\gamma q(t,x)|^2dt<\infty,\\
    \int_0^T\int_E\sup_{x\in K}|r(t,e,x)|^2v(de)dt<\infty;\\
  \end{split}
  \end{equation}

  (v) the following holds almost surely
  \begin{equation}
    \begin{split}
      p(t,x)=&\varphi(x)+\int_t^T[\mathcal {L}(s,x)p(s-,x)+\mathcal
  {M}(s,x)q(s,x)\\
  &+f(s,x,p(s,x),q(s,x)+\partial
  p(s-,x)\sigma,r(s,\cdot,\phi_{s,\cdot}(x))-p(s-,x)+p(s-,\phi_{s,\cdot}(x)))]ds\\
  &-\int_t^T\int_E[r(s,e,x)-r(s,e,\phi_{s,e}(x))+p(s-,x)-p(s-,\phi_{s,e}(x))]v(de)ds\\
  &-\int_t^Tq^{~,k}(s,x)dW^k_s-\int_t^T\int_Er(s,e,x)\widetilde{N}(deds),\ \forall
  (t,x)\in[0,T]\times\mR^n.
    \end{split}
  \end{equation}
\end{defn}
Denote
$(X_t(x),Y_t(x),Z_t(x),U_t(\cdot,x))_{t\in[0,T]}:=(X^0_t(x),Y^0_t(x),Z^0_t(x),U^0_t(\cdot,x))_{t\in[0,T]}$.
We define
\begin{equation}
\begin{split}\label{e3}
  p(t,x):=&Y_t(X_t^{-1}(x)),\\
  q(t,x):=&Z_t(X_{t-}^{-1}(x))-\pat p(t-,x)\sigma(t,x),\\
  r(t,e,x):=&p(t-,\phi_{t,e}^{-1}(x))-p(t-,x)+U_t(e,X_{t-}^{-1}(\phi_{t,e}^{-1}(x))).\\
\end{split}
\end{equation}
\begin{rem}
  At the jump time $\tau$ of the point process $\mathbbm{p}$, we have the
  relation
  $$X_\tau(x)=X_{\tau-}(x)+g(\tau,\mathbbm{p}(\tau),X_{\tau-}(x))=\phi_{\tau,\mathbbm{p}(\tau)}(X_{\tau-}(x))$$
  which implies
  $$X_{\tau-}(x)=\phi_{\tau,\mathbbm{p}(\tau)}^{-1}(X_\tau(x)).$$
  So $X_{t-}(\cdot):\mathbb{R}^n\rightarrow\mathbb{R}^n$ is also a
  homeomorphic mapping for any $t\in[0,T]$ almost surely.
\end{rem}
\begin{thm}\label{ethm1}
  Assume the same conditions for coefficients $b$, $\sigma$ and $g$ as in Theorem \ref{thmc5}. Assume further $\mathbf{(C5)_{k+1}}$ and $\mathbf{(C6)}$ are
  satisfied with $k>2+\frac{n}{2}$. Then the triple $(p,q,r)$
  defined in \eqref{e3} is a classical solution of BSIPDE
  \eqref{a1}.
\end{thm}
\begin{proof}
  From Theorem \ref{dthm3} and \eqref{cc28}, we know that the triple
  $(p,q,r)$ satisfies \eqref{e2}. From Theorem \ref{thmc5}, we have
\begin{equation}
  \left\{
  \begin{array}{ccc}
  \begin{split}\label{e4}
    dX_t^{-1}(x)=&(\mM^r\mM^r-\mL)(t,x)X_{t-}^{-1}(x)dt+\int_E\mA(t,e)X_{t-}^{-1}(x)v(de)dt\\
    &-\mM^r(t,x)X_{t-}^{-1}(x)dW^r_t+\int_E\mA(t,e)X_{t-}^{-1}(x)\widetilde{N}(dedt),\
    0\leq t\leq T,\\
    X_0^{-1}(x)=&x.\\
  \end{split}
 \end{array}
  \right.
\end{equation}
In view of Theorem \ref{dthm3} again, we apply the It\^{o}-Wentzell
formula to calculate $Y_t(X_t^{-1}(x))$ and obtain
\begin{equation}
 \left\{
  \begin{split}
    dp(t,x)=&-f(t,x,p(t,x),Z_t(X_{t-}^{-1}(x)),U_t(X_{t-}^{-1}(x)))dt+Z_t(X_{t-}^{-1}(x))dW_t\\
    &+\pat_iY_{t-}(\wtil{x})|_{\wtil{x}=X_{t-}^{-1}(x)}[(\mM^r\mM^r-\mL)(t,x)X_{t-}^{-1,i}(x)+\int_E\mA(t,e)X_{t-}^{-1,i}(x)v(de)]dt\\
    &-\pat_iY_{t-}(\wtil{x})|_{\wtil{x}=X_{t-}^{-1}(x)}\mM^r(t,x)X_{t-}^{-1,i}(x)dW^r_t\\
    &+\frac{1}{2}\pat^2_{ij}Y_{t-}(\wtil{x})|_{\wtil{x}=X_{t-}^{-1}(x)}[\mM^r(t,x)X_{t-}^{-1,i}(x)][\mM^r(t,x)X_{t-}^{-1,j}(x)]dt\\
    &-\pat_iZ^{~,r}(\wtil{x})|_{\wtil{x}=X_{t-}^{-1}(x)}\mM^r(t,x)X_{t-}^{-1,i}(x)dt\\
    &+\int_E[Y_{t-}(X_{t-}^{-1}(\phi_{t,e}^{-1}(x)))-Y_{t-}(X_{t-}^{-1}(x))-\pat_iY_{t-}(\wtil{x})|_{\wtil{x}=X_{t-}^{-1}(x)}(\mA(t,e)X_{t-}^{-1,i}(x))]v(de)dt\\
    &+\int_E[Y_{t-}(X_{t-}^{-1}(\phi_{t,e}^{-1}(x)))+U_t(e,X_{t-}^{-1}(\phi_{t,e}^{-1}(x)))-Y_{t-}(X_{t-}^{-1}(x))]\wtil{N}(dedt)\\
    &+\int_E[U_t(e,X_{t-}^{-1}(\phi_{t,e}^{-1}(x)))-U_t(e,X_{t-}^{-1}(x))]v(de)dt,\ t\in[0,T],\\
    p(T,x)=&\varphi(x)
  \end{split}
  \right.
\end{equation}
where $X_\cdot^{-1,i}(x)$ is the $i$-th component of
$X_\cdot^{-1}(x)$. By computation, we have
\begin{equation}
\begin{split}
  \mM^r(t,x)Z^{~,r}_t(X_{t-}^{-1}(x))=&\pat_iZ_t^{,r}(\wtil{x})|_{\wtil{x}=X_{t-}^{-1}(x)}\mM^r(t,x)X^{-1,i}_{t-}(x),\\
  \mM^r(t,x)p(t-,x)=&\pat_iY_{t-}(\wtil{x})|_{\wtil{x}=X_{t-}^{-1}(x)}\mM^r(t,x)X_{t-}^{-1,i}(x),\\
  \mM^r\mM^rp(t-,x)=&\mM^r(t,x)(\pat_iY_{t-}(\wtil{x})|_{\wtil{x}=X_{t-}^{-1}(x)}\mM^r(t,x)X_{t-}^{-1,i}(x))\\
  =&\mM^r(t,x)(\pat_iY_{t-}(\wtil{x})|_{\wtil{x}=X_{t-}^{-1}(x)})\mM^r(t,x)X_{t-}^{-1,i}(x)\\
  &+\pat_iY_{t-}(\wtil{x})|_{\wtil{x}=X_{t-}^{-1}(x)}(\mM^r\mM^r(t,x)X_{t-}^{-1,i}(x))\\
  =&\pat^2_{ij}Y_{t-}(\wtil{x})|_{\wtil{x}=X_{t-}^{-1}(x)}[\mM^r(t,x)X_{t-}^{-1,i}(x)][\mM^r(t,x)X_{t-}^{-1,j}(x)]\\
  &+\pat_iY_{t-}(\wtil{x})|_{\wtil{x}=X_{t-}^{-1}(x)}(\mM^r\mM^r(t,x)X_{t-}^{-1,i}(x)),\\
  \mL(t,x)p(t-,x)=&\pat_iY_{t-}(\wtil{x})|_{\wtil{x}=X_{t-}^{-1}(x)}\mL(t,x)X_{t-}^{-1,i}(x)\\
  &+\frac{1}{2}\pat^2_{ij}Y_{t-}(\wtil{x})|_{\wtil{x}=X_{t-}^{-1}(x)}[\mM^r(t,x)X_{t-}^{-1,i}(x)][\mM^r(t,x)X_{t-}^{-1,j}(x)].\\
\end{split}
\end{equation}
So the triple $(p,q,r)$ is a classical solution to BSIPDE
\eqref{a1}.
\end{proof}
The following theorem is concerned with the uniqueness of the
classical solution to the system \eqref{a1}.
\begin{thm}\label{ethm2}
  Let the assumptions of Theorem \ref{ethm1} be satisfied. Let $(\wtil{p},\wtil{q},\wtil{r})$ be a classical solution to BSIPDE \eqref{a1}. Then we have
  \begin{eqnarray*}
    &\wtil{p}(t,X_t(x))=Y_t(x),&\\
    &\wtil{q}(t,X_{t-}(x))=Z_t(x)-\pat
    \wtil{p}(t-,X_{t-}(x))\sigma(t,X_{t-}(x)),&\\
    &\wtil{r}(t,e,\phi_{t,e}(X_{t-}(x)))=U_t(e,x)+\wtil{p}(t-,X_{t-}(x))-\wtil{p}(t-,\phi_{t,e}(X_{t-}(x)))&
  \end{eqnarray*}
  or equivalently,
  \begin{eqnarray*}
      &\wtil{p}(t,x)=Y_t(X_t^{-1}(x)),&\\
      &\wtil{q}(t,x)=Z_t(X_{t-}^{-1}(x))-\pat[Y_{t-}(X_{t-}^{-1}(x))]\sigma(t,x),&\\
      &\wtil{r}(t,e,x)=Y_{t-}(X_{t-}^{-1}(\phi_{t,e}^{-1}(x)))-Y_{t-}(X_{t-}^{-1}(x))+U_t(e,X_{t-}^{-1}(\phi_{t,e}^{-1}(x))).&
  \end{eqnarray*}
\end{thm}
\begin{proof}
  Using the It\^{o}-Wentzell formula to calculate
  $\wtil{p}(t,X_t(x))$, we see that
  \begin{equation*}
   \begin{split}
    &(\wtil{p}(t,X_t(x)),\wtil{q}(t-,X_{t-}(x))+\pat
    \wtil{p}(t-,X_{t-}(x))\sigma(t,X_{t-}(x)),\\
    &\wtil{r}(t,e,\phi_{t,e}(X_{t-}(x)))-\wtil{p}(t-,X_{t-}(x))+\wtil{p}(t-,\phi_{t,e}(X_{t-}(x))))
   \end{split}
  \end{equation*}
    is an adapted solution of BSDE \eqref{d10}. We have the desired
    result.
\end{proof}



\end{document}